\documentclass[12pt,a4paper]{amsart}
\usepackage[a4paper, left=28mm, right=28mm, top=28mm, bottom=34mm]{geometry}
\usepackage{latexsym}
\usepackage{amsmath}
\usepackage{amssymb}
\usepackage{amsthm}
\usepackage{amscd}
\usepackage{mathrsfs}
\usepackage{mathtools}
\usepackage[all]{xy}
\usepackage{graphicx}
\usepackage{comment}
\usepackage{arydshln}

\newtheorem{theorem}{Theorem}[section]
\newtheorem{lemma}[theorem]{Lemma}
\newtheorem{corollary}[theorem]{Corollary}
\newtheorem{proposition}[theorem]{Proposition}
\newtheorem{condition}[theorem]{Condition}
\newtheorem{conjecture}[theorem]{Conjecture}

\newtheorem{example}[theorem]{Example}
\theoremstyle{definition}
\newtheorem{question}[theorem]{Question}

\newtheorem{remark}[theorem]{Remark}

\theoremstyle{remark}

\makeatletter

\@addtoreset{equation}{section}
\makeatother

\makeatletter

\@addtoreset{equation}{section}
\makeatother

\makeatletter
\newcommand{\subsubsubsection}{\@startsection{paragraph}{4}{\z@}%
 {1.0\Cvs \@plus.5\Cdp \@minus.2\Cdp}%
 {.1\Cvs \@plus.3\Cdp}%
 {\reset@font\sffamily\normalsize}
 }
\makeatother
\DeclareMathOperator{\Spec}{Spec}

\DeclareMathOperator{\Hom}{Hom}
\DeclareMathOperator{\Ker}{Ker}

\DeclareMathOperator{\Tr}{Tr}

\DeclareMathOperator{\Ima}{Im}
\DeclareMathOperator{\la}{\lambda}

\usepackage{bm}

\usepackage[usenames,dvipsnames]{color}

\begin{document}

\title[Maximality and minimality of van der Geer--van der Vlugt curves]{Criteria of maximality and minimality of van der Geer--van der Vlugt curves}

\author{Tetsushi Ito}
\address{
Department of Mathematics, Faculty of Science, Kyoto University
Kyoto, 606--8502, Japan}
\email{tetsushi@math.kyoto-u.ac.jp}

\author{Ren Tatematsu}
\address{
Department of Mathematics, Faculty of Science, Kyoto University
Kyoto, 606--8502, Japan}
\email{tatematsu.ren.68x@st.kyoto-u.ac.jp}

\author{Takahiro Tsushima}
\address{
Keio University School of Medicine,
4-1-1 Hiyoshi, Kohoku-ku,
Yokohama, 223-8521, Japan}
\email{tsushima@keio.jp}

\date{December 19, 2024}



\begin{abstract}
The van der Geer--van der Vlugt curves are Artin--Schreier coverings of the affine line
defined by linearized polynomials over finite fields.
We give several criteria for them to be maximal or minimal, i.e.\ attaining the upper or lower bound in the Hasse--Weil inequalities.
We also study the $L$-polynomials of certain generalizations of van der Geer--van der Vlugt curves.
As applications, we find several maximal (or minimal) curves among them.
Our proof is based on an explicit formula of $L$-polynomials recently obtained by 
Takeuchi and the third author.
\end{abstract}

\maketitle

\section{Introduction}\label{Introduction}

Let $C$ be a geometrically connected smooth 
projective curve of genus $g(C)$
over the finite field $\mathbb{F}_q$
of cardinality $q$.
We have the Hasse--Weil inequalities:
\[
q + 1 - 2 g(C) q^{1/2} \, \leq \, |C(\mathbb{F}_q)| \, \leq \, q + 1 + 2 g(C) q^{1/2}.
\]
The curve $C$ is called \emph{maximal} over $\mathbb{F}_q$ (or \emph{$\mathbb{F}_q$-maximal})
if $|C(\mathbb{F}_q)|=q+1+2 g(C) q^{1/2}$ holds.
Similarly, it is called 
\emph{minimal} over $\mathbb{F}_q$ (or \emph{$\mathbb{F}_q$-maximal})
if $|C(\mathbb{F}_q)|=q+1-2 g(C) q^{1/2}$) holds. 
Finding explicit maximal (or minimal) curves is an interesting problem,
which has important applications to coding theory and number 
theory; see \cite{St} for backgrounds.

In this paper, we focus on a particular class of curves called the van der Geer--van der Vlugt curves,
which are Artin--Schreier coverings of the affine line.
We fix a prime number $p_0$.
Let $p$ be a power of $p_0$, and $q$ a power of $p$.
We write $q = p_0^{f_0}$.
Let
\[
  R(x)=\sum_{i=0}^e a_i x^{p^i} \in 
\mathbb{F}_q[x]
  \qquad (a_e \neq 0)
\] 
be an $\mathbb{F}_p$-linearized polynomial of degree $p^e$ over $\mathbb{F}_q$.
Let $C_R$ be the smooth affine curve over $\mathbb{F}_q$
defined by 
\[
y^p-y=xR(x).
\]
The smooth compactification $\overline{C}_R$ of $C_R$ is
a geometrically connected smooth projective curve over $\mathbb{F}_q$ with genus $p^e(p-1)/2$.
It is called the \emph{van der Geer--van der Vlugt curve} associated with $R(x)$. 
When $p=p_0$ (i.e.\ $p$ is a prime number),
van der Geer--van der Vlugt studied the automorphism group and the $L$-polynomial of $\overline{C}_R$ in \cite{GV};
see also \cite{BHMSSV}.
For general $p$, some examples of maximal (or minimal) van der Geer--van der Vlugt curves are found by several people;
see \cite{B}, \cite{CO}, \cite{C}, \cite{OS}.

When $p_0 \neq 2$, Takeuchi and the third author
obtained an explicit formula of the $L$-polynomials of Geer--van der Vlugt curves satisfying some conditions in terms of quadratic Gauss sums,
and prove for the maximality (or minimality) over
$\mathbb{F}_{q^{2 p_0}}$ (or $\mathbb{F}_{q^{4 p_0}}$);
see \cite[Theorem 1.2]{TT}.
However, the maximality (or minimality) over $\mathbb{F}_{q^{2 p_0}}$ (or $\mathbb{F}_{q^{4 p_0}}$)
does not imply 
the maximality (or minimality) over a smaller field;
see Lemma \ref{maximal minimal extension}.
The results in \cite{TT} do not say anything about
the maximality (or minimality) over $\mathbb{F}_q$ (or $\mathbb{F}_{q^2}$).

The aim of this paper is to give 
refinements of the results in \cite{TT} when $p_0 \neq 2$.
We give several criteria for the maximality and minimality over $\mathbb{F}_q$ (or $\mathbb{F}_{q^2}$) of van der Geer--van der Vlugt curves satisfying some conditions;
see Theorem \ref{cpq}, Theorem \ref{ttbb}, Theorem \ref{ttb3} and Theorem \ref{ttb4}.
We also study generalizations of van der Geer--van der Vlugt curves defined by the equation
\[
y^{p^r}-y=xR(x)
\]
for some $r \geq 1$.
As applications,
we find several examples of maximal and minimal
curves among them.
To the best of the authors' knowledge,
some of the curves found in this paper were not studied previously in the literature.

In order to illustrate our results,
we give some examples of curves to which our criteria can be applied.

In the following theorem, we establish the maximality and minimality of van der Geer--van der Vlugt curves
arising from linearized polynomials satisfying certain
symmetry conditions.

\begin{theorem}[see Theorem \ref{214}]
\label{214t}
Assume $p_0 \neq 2$, $q = p^n$, and $p \equiv 1 \pmod n$.
Let $e \geq 1$, and $c_0$, $c_1$, $\ldots$, $c_{e-1} \in \mathbb{F}_p$ be elements such that the polynomial 
 \[
        g(x) \coloneqq x^{2e} + c_{e - 1} x^{2 e - 1} + \cdots +  c_1 x^{e + 1} + c_0 x^e + c_1 x^{e - 1}
            + \cdots + c_{e - 1} x + 1 \in \mathbb{F}_p[x]
            \]
            divides $x^n-1$. 
            Let 
              \[
        R(x) \coloneqq 2 x^{p^{e}} + 2 c_{e - 1} x^{p^{e - 1}} + \cdots + 2 c_1 x^p + c_0 x
        \in \mathbb{F}_p[x]. 
    \]
Let $\overline{C}_R$ be the van der Geer--van der Vlugt curve over $\mathbb{F}_q$ associated with $R(x)$.
\begin{itemize}
\item[{\rm (1)}] 
The curve $\overline{C}_R$ is
$\mathbb{F}_{q^2}$-maximal if $q \equiv 3 \pmod 4$. 
It is
$\mathbb{F}_{q^2}$-minimal 
if $q \equiv 1 \pmod 4$.

\item[{\rm (2)}]
  Recall that $q = p_0^{f_0}$.
   \begin{itemize}
   \item[{\rm (i)}]
   The curve $\overline{C}_R$ is 
   $\mathbb{F}_q$-maximal if and only if
   $n$ is even, $p_0 \equiv 3 \pmod 4$ and  
    $f_0 \equiv 2 \pmod 4$.

  \item[{\rm (ii)}]
   The curve $\overline{C}_R$ is 
   $\mathbb{F}_q$-minimal if and only if one of the following conditions holds:
   \begin{itemize}
   \item[$\bullet$] $n$ is even, and $p_0 \equiv 1 \pmod 4$, 
   \item[$\bullet$] $n$ is even, $p_0 \equiv 3 \pmod 4$, and  
    $f_0 \equiv 0 \pmod 4$.
   \end{itemize}
   \end{itemize}

\item[{\rm (3)}]
   If $n$ is odd, the curve 
   $\overline{C}_R$ is neither  
   $\mathbb{F}_q$-maximal nor $\mathbb{F}_q$-minimal.
\end{itemize}
\end{theorem}

\begin{remark}
\label{bqsy}
The curve $\overline{C}_R$ in Theorem \ref{214t} is isomorphic to the curve $\mathcal{X}_{f,k}$ for $s = 0$
studied by Bartoli--Quoos--Sayg\i--Y\i lmaz in \cite[Theorem 6]{B}.
They proved that it is either $\mathbb{F}_q$-maximal or $\mathbb{F}_q$-minimal, but they did not determine
when $\overline{C}_R$ is $\mathbb{F}_q$-maximal or $\mathbb{F}_q$-minimal.
Their methods and our methods are different.
\end{remark}

In the following theorem,
we find explicit maximal (or minimal)
van der Geer--van der Vlugt curves
associated with
$R(x)=2x^p+x \in \mathbb{F}_p[x]$.

\begin{theorem}[see Theorem \ref{lc}]\label{ata}
Assume $p_0 \neq 2$.
Let $R(x)=2x^p+x \in \mathbb{F}_p[x]$.
Let $k \geq 1$ be a positive integer.
\begin{itemize}
\item[{\rm (1)}] 
Assume $p \equiv 1 \pmod 4$. 
The curve $\overline{C}_R$ is 
$\mathbb{F}_{p^k}$-minimal if $k \equiv 0 \pmod{6}$. 

\item[{\rm (2)}] 
Assume $p \equiv 3 \pmod 4$.
The curve $\overline{C}_R$ is 
$\mathbb{F}_{p^k}$-maximal if and only if $k \equiv 0 \pmod{6}$ and $k/6$ is odd. 
\end{itemize}
\end{theorem}

\begin{remark}
When $p_0=3$ or  $p_0 \equiv 7 \pmod{12}$, the `if' part in 
Theorem \ref{ata} was proved by \"Ozbudak--Sayg\i \ by a different method; see \cite[Theorem 3.10]{OS}.
\end{remark}

We consider a family of twists
of van der Geer--van der Vlugt curves
$\{\overline{C}_{\zeta R}\}_{\zeta \in\mathbb{F}^{\times}}$,
where $\mathbb{F}$ is an algebraic closure of $\mathbb{F}_p$.
Among the twists $\{\overline{C}_{\zeta R}\}_{\zeta \in\mathbb{F}^{\times}}$,
we find several maximal curves. 
Here we state one of our results.

\begin{theorem}[see Theorem \ref{mp}]
Assume $p_0 \neq 2$. Let $R(x)=2x^p+x$.
Let $d$ be the order of $-2$ in $\mathbb{F}_{p_0}^{\times}$,
and let $\zeta\in \mathbb{F}_{p^d}^{\times}$ be 
an element satisfying $\zeta^{p-1} = -2^{-1}$.
If $d$ is odd and $p \equiv 3 \pmod 4$,
the curve $\overline{C}_{\zeta R}$ is $\mathbb{F}_{p^{2p_0d}}$-maximal.
\end{theorem}

Furthermore, we also study generalizations of the van der Geer--van der Vlugt curves.
Let $R(x) \in \mathbb{F}_{p^{2f}}[x]$ be an $\mathbb{F}_p$-linearized polynomial 
of degree $p^e$ over $\mathbb{F}_{p^{2f}}$.
For a positive integer $r \geq 1$,
consider the smooth affine curve $C_{R,r}$ defined by
\[ z^{p^r}-z=xR(x). \]
Let $\overline{C}_{R,r}$ be the smooth compactification of $C_{R,r}$.

In Section \ref{Generalization}, we study the $L$-polynomial of $C_{R,r}$ and its relation to
the twists of van der Geer--van der Vlugt curves.

Assume $e \geq 1$, and let
$e_R \coloneqq \mathrm{gcd}\{1 \leq i \leq e \mid a_i \neq 0\}$.
We conjecture that if $\overline{C}_{R,r}$ is $\mathbb{F}_{p^{2f}}$-maximal,
then $r$ divides $\mathrm{gcd}(e_R, f)$;
see Conjecture \ref{conj}.
We also give several evidences.
If $R(x)$ is a monomial, this conjecture follows from Coulter's results \cite[Theorem 8.12]{C}. 
When $p_0 \neq 2$, we give some evidences of this conjecture for $R(x)=2x^p+x$
using maximal twists $\overline{C}_{\zeta R}$ 
mentioned above. 
For example, if $p_0 \equiv 5 \pmod{8}$ and
$r$ is the order of $-2$ in $\mathbb{F}_{p_0}^{\times}$,
we show that the curve 
$\overline{C}_{R,r}$ is not $\mathbb{F}_{p_0^k}$-maximal 
for any $k \geq 1$;
see Corollary \ref{lcc2} (1).

\subsubsection*{Organization of this paper}

In Section \ref{maximality and minimality},
we summarize known results on the maximality and minimality of curves over finite fields.
In Section \ref{Explicit Formula},
we recall an explicit formula of the $L$-polynomials of van der Geer--van der Vlugt curves
proved by Takeuchi and the third author \cite{TT}.
In Section \ref{Criteria}, we prove criteria on maximality and minimality.
In Section \ref{Examples maximal van der Geer--van der Vlugt curves},
we give some examples of maximal and minimal van der Geer--van der Vlugt curves,
and prove Theorem \ref{214t}. 

In Section \ref{Generalization}, we study
the curves $\overline{C}_{R,r}$,
which is a generalization of the van der Geer--van der Vlugt curves.
We give a formula for  the $L$-polynomial of $\overline{C}_{R,r}$ using a formula 
given in \cite[Theorem 1.1]{TT}.
As a result, we show that  $\overline{C}_{R,r}$ is supersingular.
We give a 
necessary and sufficient condition for $\overline{C}_{R,r}$ to be maximal in a certain case, 
which is regarded as a generalization of
\"Ozbudak--Sayg\i 's result \cite[Proposition 2.1]{OS}. 

Finally, in Section \ref{Further examples}, we give a criterion for the twist 
$\overline{C}_{\zeta R}\ (\zeta \in \mathbb{F}^{\times})$
with $R(x)=2x^p+x$ to be maximal. 
In some special cases, we find new maximal 
curves among them. We show 
Theorem \ref{ata}. 
Using these examples, in certain cases, 
we show that $\overline{C}_{R,r}$ is not maximal 
over any finite extension of $\mathbb{F}_{p_0}$.

\subsubsection*{Notation}

We use the same notation as in \cite{TT}.
Throughout the paper,
we fix a prime number $p_0$. Let $p$ be a power of $p_0$, and $q$ a power of $p$.
We write $q = p_0^{f_0}$.
We fix an algebraic closure $\mathbb{F}$ of $\mathbb{F}_{p_0}$,
and consider finite extensions of $\mathbb{F}_{p_0}$ (such as $\mathbb{F}_p$ or $\mathbb{F}_q$) as subfields of $\mathbb{F}$.

For a power $q$ of $p$,
we write $\Tr_{q/p}$ for the trace map 
$\mathrm{Tr}_{\mathbb{F}_{q}/\mathbb{F}_p} \colon \mathbb{F}_{q} \to \mathbb{F}_p$.

A polynomial $f(x) \in \mathbb{F}_p[x]$ 
is called \emph{additive} if $f(x+y) = f(x) + f(y)$. 
For an integer $k \geq 1$, an additive polynomial 
$f(x)$ is called \emph{$\mathbb{F}_{p_0^k}$-linearized} 
if $f(s x) = s f(x)$ for every 
$s \in \mathbb{F}_{p_0^k}$. 
For an additive polynomial 
$f(x) \in \mathbb{F}_p[x]$, 
we put $\Ker f \coloneqq \{x \in \mathbb{F} \mid 
f(x)=0\}$. 
If $f(x)$ is an $\mathbb{F}_{p_0^k}$-linearized polynomial, 
$\Ker f$ is regarded as an 
$\mathbb{F}_{p_0^k}$-vector space. 

When $p_0 \neq 2$,
we put $\displaystyle \left (\frac{x}{q} \right) \coloneqq x^{(q-1)/2} \in \{ 1, -1 \}$
for an element $x \in \mathbb{F}_q^{\times}$.
When $q$ is a prime number,
it is nothing but the Legendre symbol.

\section{The maximality and minimality of curves over finite fields}
\label{maximality and minimality}

In this section, we summarize known results on the maximality and minimality of curves over finite fields.

As in Introduction, let $p_0$ be a prime number,
and $p$ be a power of $p_0$.
We take a prime number 
$\ell \neq p_0$ and 
use the following notation on $\ell$-adic cohomology.
For a variety $X$ over $\mathbb{F}_p$, let 
 $X \otimes_{\mathbb{F}_p} \mathbb{F}$ denote the 
 base change of $X$ to $\mathbb{F}$.
For a $\overline{\mathbb{Q}}_{\ell}$-sheaf $\mathscr{F}$ on $X$, we simply write
\[
  H_{\rm c}^i(X,\mathscr{F})
  \coloneqq
  H_{\text{\'et,c}}^i(X\otimes_{\mathbb{F}_p} \mathbb{F},\mathscr{F}).
  \]
When $\mathscr{F} = \overline{\mathbb{Q}}_{\ell}$,
we simply write
\[
H_{\rm c}^i(X) \coloneqq
H_{\text{\'et}, \rm c}^i(X \otimes_{\mathbb{F}_p} \mathbb{F},\overline{\mathbb{Q}}_{\ell}),
\qquad
H^i(X) \coloneqq
H_{\text{\'et}}^i(X \otimes_{\mathbb{F}_p} \mathbb{F},\overline{\mathbb{Q}}_{\ell}).
\]
For a morphism between $\mathbb{F}_p$-varieties  
$f \colon X \to Y$, let $f^\ast \colon 
H^i(Y) \to H^i(X)$ denote the pull-back by $f$. 
Similarly, for a proper morphism between $\mathbb{F}_p$-varieties  
$f \colon X \to Y$, let $f^\ast \colon 
H^i_{\rm c}(Y) \to H^i_{\rm c}(X)$ denote the pull-back by $f$. If $f$ is finite, $f^\ast$ is an injection. 

Let $\mathrm{Fr}_p \colon \mathbb{F} \to \mathbb{F};\ 
 x \mapsto x^{p^{-1}}$ denote the geometric Frobenius automorphism. 
There is a natural action of $\mathrm{Fr}_p$ on
$H_{\rm c}^i(X,\mathscr{F})$,
$H_{\rm c}^i(X)$, and $H^i(X)$.

Let $C$ be a geometrically connected smooth projective curve of genus $g(C)$ over $\mathbb{F}_p$.
The \emph{$L$-polynomial} of $C$ over $\mathbb{F}_p$
is defined by
  \[
 L_{C/\mathbb{F}_p}(T)\coloneqq
 \det(1-\mathrm{Fr}_p^\ast T; 
 H^1(C)).
 \]
Let $\alpha_1,\ldots,\alpha_{2 g(C)} \in \overline{\mathbb{Q}}_{\ell}$ be the eigenvalues of $\mathrm{Fr}_p^\ast$ on
$H^1(C)$.

The following results are well-known.
We omit the proofs.

\begin{lemma}
\label{maximal minimal}
\begin{itemize}
\item[{\rm (1)}] The curve $C$ is supersingular (i.e.,\ over $\mathbb{F}$, the Jacobian variety $\mathrm{Jac}(C)$ is isogenous to a power of a supersingular elliptic curve)
if and only if $\alpha_i/p^{1/2}$ is a root of unity for every $i$ $(1 \leq i \leq 2 g(C))$.

\item[{\rm (2)}] The curve $C$ is $\mathbb{F}_p$-maximal (resp.\ $\mathbb{F}_p$-minimal)
if and only if $\alpha_i  = - p^{1/2}$ (resp.\ $\alpha_i  = p^{1/2}$) for every $i$ $(1 \leq i \leq 2 g(C))$.
\end{itemize}
\end{lemma}

\begin{lemma}
\label{maximal minimal extension}
\begin{itemize}
\item[{\rm (1)}] If the curve $C$ is $\mathbb{F}_p$-minimal,
then it is $\mathbb{F}_{p^k}$-minimal for every integer $k \geq 1$.

\item[{\rm (2)}] If the curve $C$ is $\mathbb{F}_p$-maximal,
then it is $\mathbb{F}_{p^{2k+1}}$-maximal and $\mathbb{F}_{p^{2k}}$-minimal for every integer $k \geq 1$.
\end{itemize}
\end{lemma}

The following result will be used in Section \ref{Further examples}.

\begin{lemma}\label{crl}
Assume that there exist finite morphisms between geometrically connected smooth 
projective curves $C \to C_1$
over $\mathbb{F}_{p^{n_1}}$
and $C \to C_2$ 
over $\mathbb{F}_{p^{n_2}}$ such that
the following conditions are satisfied:
\begin{itemize}
\item[{\rm (1)}]
The curve $C_1$ is $\mathbb{F}_{p^{n_1}}$-maximal or $\mathbb{F}_{p^{n_1}}$-minimal.

\item[{\rm (2)}]
The curve $C_2$ is $\mathbb{F}_{p^{n_2}}$-maximal.

\item[{\rm (3)}]
$n_1 \not\equiv 0 \pmod{4}$ and $n_2 \equiv 0 \pmod{4}$.
\end{itemize}
Then $C$ is not $\mathbb{F}_{p^k}$-maximal for
any positive integer $k \geq 1$.  
\end{lemma}

\begin{proof}
Let $\alpha_1,\ldots,\alpha_{2 g(C)} \in \overline{\mathbb{Q}}_{\ell}$ be the eigenvalues of $\mathrm{Fr}_p^\ast$ on
$H^1(C)$.
We put $\zeta_i \coloneqq \alpha_i/p^{1/2}$.
The condition (1) implies $\zeta_{i_1}^{n_1} \in \{ \pm 1 \}$ for some $i_1$,
and the condition (2) implies $\zeta_{i_2}^{n_2} = -1$
for some $i_2$.
Permuting $\alpha_1,\ldots,\alpha_{2 g(C)}$ if necessary, we may assume $\zeta_1^{n_1} \in \{ \pm 1 \}$ and $\zeta_2^{n_2} = -1$.

For an integer $r \geq 1$, 
we write $\mu_r$ for the set of the 
$r$-th roots of unity in $\overline{\mathbb{Q}}_{\ell}$. 
By the condition (3), we can write 
$n_1 = 2^r n'_1$ with $r \in \{ 0, 1\}$ and an odd integer $n'_1$,
and $n_2 = 2^s n'_2$ with $s \geq 2$ and an odd integer $n'_2$.
By 
$\zeta_1^{n_1}=\pm 1$, 
we may write 
$\zeta_1=\zeta'_1 \zeta''_1$ with some 
$\zeta'_1 \in \mu_4$ and 
$\zeta''_1 \in \mu_{n'_1}$. 
Similarly, by 
$\zeta_2^{n_2}=-1$, we can write  
$\zeta_2=\zeta'_2 \zeta''_2$, where  
$\zeta'_2$ is a primitive $2^{s+1}$-th root of unity and 
$\zeta''_2 \in \mu_{n'_2}$.

Assume that $C$ is $\mathbb{F}_{p^k}$-maximal for some $k \geq 1$.
Then $\zeta_1^k = \zeta_2^k = -1$.
By $\zeta_2^k = -1$, we have $k \equiv 0 \pmod{2^s}$. 
Since $s \geq 2$, we have $\zeta_1^k = \zeta''^k_1
\in \mu_{n'_1}$.
Since $n'_1$ is odd, we have $\zeta_1^k \neq -1$.
This contradicts $\zeta_1^k=-1$. Hence the claim follows. 
%\end{comment}
\end{proof}

\section{An explicit formula of the $L$-polynomials of van der Geer--van der Vlugt curves}
\label{Explicit Formula}

In this section, we recall
an explicit formula of the $L$-polynomials of van der Geer--van der Vlugt curves
obtained by Takeuchi and the third author in \cite{TT}.

\subsection{Heisenberg groups arising from linearized polynomials}
\label{Heisenberg group}

We keep the same notation as in Section \ref{Introduction}.
Let
\[
  R(x)=\sum_{i=0}^e a_i x^{p^i} \in 
\mathbb{F}_q[x]
  \qquad (a_e \neq 0)
\] 
be an $\mathbb{F}_p$-linearized polynomial of degree $p^e$ over $\mathbb{F}_q$.
We put 
\begin{align}
\label{er}
E_R(x) &\coloneqq R(x)^{p^e}+\sum_{i=0}^e (a_i x)^{p^{e-i}} \in \mathbb{F}_q[x],  \\
\label{fr}
f_R(x,y) &\coloneqq - \sum_{i = 0}^{e - 1}
        \left(
            \sum_{j = 0}^{e - i - 1} (a_i x^{p^i} y)^{p^j} + (xR(y))^{p^i}
                \right) \in \mathbb{F}_q[x, y], \\
V_R &\coloneqq \Ker E_R = \{x \in \mathbb{F} \mid E_R(x)=0\}. \notag
\end{align}
Since $E_R(x)$ is a separable $\mathbb{F}_p$-linearized polynomial of degree $p^{2e}$,
$V_R$ is an $\mathbb{F}_p$-vector space of dimension $2e$.
By \cite[(2.1)]{TT}, we have the following equality:
\begin{equation}
\label{11}
f_R(x,y)^p-f_R(x,y) = -x^{p^e} E_R(y)+x R(y)+y R(x). 
\end{equation}

We define a \emph{Heisenberg group} $H_R$ by
\[
H_R \coloneqq \{(\alpha,\beta) \in \mathbb{F}^2 \mid \alpha\in V_R,\ 
\beta^p-\beta=\alpha R(\alpha)\},
\]
whose group operation is given by
\begin{equation}\label{op}
(\alpha,\beta) \cdot (\alpha',\beta')=
(\alpha+\alpha', \beta+\beta'+f_R(\alpha,\alpha')).
\end{equation}
The center of $H_R$ equals $\{0\} \times 
\mathbb{F}_p$, which we identify with 
$\mathbb{F}_p$ via the projection
$(0,\beta) \mapsto \beta$. 

Let $A \subset H_R$ be a maximal abelian subgroup.
Let 
 \[
 \pi \colon H_R \to V_R;\ 
 (\alpha,\beta) \mapsto \alpha
 \]
be the projection onto the first factor.
The map $\pi$ induces an isomorphism $H_R/\mathbb{F}_p \xrightarrow{\sim} V_R$. 
 By \cite[Lemma 2.5]{TT}, there exists a monic $\mathbb{F}_p$-linearized polynomial
\[
F_A(x)=\sum_{i=0}^e b_i x^{p^i} \in \mathbb{F}_q[x]
\quad
(b_e = 1)
\]
 such that 
 $A=\pi^{-1}(\overline{A})$, 
 where we set  
\[ \overline{A} \coloneqq \Ker F_A=\{x \in \mathbb{F} \mid 
 F_A(x)=0\}. \]
Note that $\overline{A}$ is an $\mathbb{F}_p$-vector space of dimension $e$.

 In the following, we assume
 $p_0 \neq 2$, $e \geq 1$, and $A\subset \mathbb{F}_q^2$.
Then the assumption $A \subset \mathbb{F}_q^2$
is satisfied if $\mathbb{F}_q$ is a splitting field of $F_A(x)$; see \cite[Lemma 2.6]{TT}. 

 We put 
 \begin{equation}\label{cc1}
 c_A \coloneqq (-1)^{e} \frac{a_e}{2} \prod_{\alpha
 \in \overline{A} \setminus \{0\}} \alpha^{-1}
 = (-1)^e \frac{a_e}{2 b_0} \quad \in \mathbb{F}_q. 
 \end{equation}

Let $a(x) \in \mathbb{F}_q[x]$ be a unique $\mathbb{F}_p$-linearized polynomial satisfying
 \begin{equation}\label{af}
  a(F_A(x)) = F_A(a(x)) = x^q-x. 
  \end{equation} 
We consider the polynomial $a(x)$ as the map
\[
  a \colon \mathbb{F}_q \to
  \overline{A};\ t \mapsto a(t).
\]
By \cite[Remark 2.9]{TT},
we embed $\overline{A}$ into $A$
by the homomorphism
 \[
\phi_{\overline{A}} \colon 
\overline{A} \to A;\ 
x \mapsto (x,f_R(x,x)/2).
 \]

 \begin{remark}
 There is a sign error in 
 \cite[Lemma A.6]{TT}. 
 The correct formula for the constant $c_A$ is
 $(-1)^{e} (a_e b_e)/(2b_0)$,
 not 
 $(-1)^{e+1} (a_e b_e)/(2b_0)$. 
\end{remark}

\subsection{$L$-polynomials of van der Geer--van der Vlugt curves}

%We take a prime number $\ell \neq  p_0$. 
For a finite abelian group $G$,
the character group with values in $\overline{\mathbb{Q}}_{\ell}^{\times}$ is written as
\[
G^{\vee} \coloneqq \Hom_{\mathbb{Z}}(G,\overline{\mathbb{Q}}_{\ell}^{\times}).
\]

For a non-trivial character
$\psi \in \mathbb{F}_p^{\vee} \setminus \{1\}$,
we put
\[
 \psi_q \coloneqq \psi \circ 
 \Tr_{q/p} \colon \mathbb{F}_q \to \overline{\mathbb{Q}}_{\ell}^{\times}
\]
and
\[
 A_{\psi}^{\vee} \coloneqq \{\xi \in A^{\vee} \mid \xi|_{\mathbb{F}_p}=\psi\}.
\]
For a character $\xi \in 
 A_{\psi}^{\vee}$,
we define a character $\xi' \in \mathbb{F}_q^{\vee}$ by
 \begin{equation}\label{xi}
 \xi' \coloneqq \xi \circ \phi_{\overline{A}} \circ a \colon
 \mathbb{F}_q \overset{a}{\longrightarrow} \overline{A} \overset{\phi_{\overline{A}}}{\longrightarrow} A \overset{\xi}{\longrightarrow} \overline{\mathbb{Q}}_{\ell}^{\times}.
 \end{equation}
Then there exists a unique element
 $\eta \in \mathbb{F}_q^{\times}$
 such that 
$\xi'(x) = \psi_q(\eta x)$
 for every $x \in \mathbb{F}_q$.

 We define the \emph{quadratic Gauss sum} associated with $\psi_q \in \mathbb{F}_q^{\vee}$ by
 \[
 G(\psi_q) \coloneqq -\sum_{x \in \mathbb{F}_q}
 \psi_q(x^2). 
 \]
Recall that $q = p_0^{f_0}$.
It is well-known that
\begin{equation}\label{gauss sum square}
G(\psi_q)^2 = (-1)^{f_0} q = \left(\frac{-1}{q}\right) q
= \begin{cases}
q & (q \equiv 1 \pmod{4}), \\
-q & (q \equiv 3 \pmod{4}).
\end{cases}
\end{equation}
(See \cite[Theorem 5.12 (iv) and Remark 5.13]{Lidl-Niederreiter}.)
In particular,
the value of $G(\psi_q)^2$ does not depend on the choice of $\psi_q$.

The \emph{$L$-polynomial}
of the van der Geer--van der Vlugt curve $\overline{C}_R$ over $\mathbb{F}_q$
is defined by
  \[
 L_{\overline{C}_R/\mathbb{F}_q}(T)\coloneqq
 \det(1-\mathrm{Fr}_q^\ast T; 
 H^1(\overline{C}_R)).
 \]
  
 The following formula was proved by Takeuchi and the third author.
 
 \begin{theorem}[{\cite[Theorem 1.1 and Corollary A.8]{TT}}]
 \label{tt}
 We assume $p_0 \neq 2$, $e \geq 1$, and 
 $A \subset \mathbb{F}_q^2$.
\begin{itemize}
\item[{\rm (1)}] 
 We have the equality 
 \[
 L_{\overline{C}_R/\mathbb{F}_q}(T)
 =\prod_{\psi \in \mathbb{F}_p^{\vee} \setminus \{1\}}
 \prod_{\xi \in A_{\psi}^{\vee}}(1-\tau_{\xi}T),  
 \] 
 where 
 \[
 \tau_{\xi} =
 \psi_q(- 4^{-1} c_A^{-1} \eta^2)
 \left(\frac{c_A}{q}\right) G(\psi_q) \quad 
 \textrm{for $\psi \in \mathbb{F}_p^{\vee} \setminus \{1\}$ and $\xi \in A_{\psi}^{\vee}$}. 
 \]

\item[{\rm (2)}] 
Let $k \geq 1$ be a positive integer.
The curve $\overline{C}_R$ is
$\mathbb{F}_{q^k}$-maximal (resp.\ $\mathbb{F}_{q^k}$-minimal)
if and only if
$\tau_{\xi}^k = - q^{k/2}$
(resp.\ $\tau_{\xi}^k = q^{k/2}$)
for every $\psi \in \mathbb{F}_p^{\vee} \setminus \{1\}$ and
every $\xi \in A_{\psi}^{\vee}$. 
\end{itemize}
\end{theorem}

\begin{proof}
(1) In \cite[Definition 3.5]{TT},
the number $\tau_{\xi} \in \overline{\mathbb{Q}}_{\ell}^{\times}$ is defined as the eigenvalue of $\mathrm{Fr}_q^\ast$ acting on the $\ell$-adic cohomology with compact support
$H^1_{\text{\rm c}}(\mathbb{A}^1, \mathscr{Q}_{\xi})$.
By \cite[Corollary 3.6]{TT},
the $L$-polynomial $L_{\overline{C}_R/\mathbb{F}_q}(T)$
is equal to the product of $1-\tau_{\xi}T$.
When $p_0 \neq 2$, an explicit formula of $\tau_{\xi}$ in terms of the quadratic Gauss sum is obtained in \cite[Corollary A.8]{TT}.
Combining these results, we get (1). %as in  Theorem \ref{tt}.

(2) The claim follows from (1) and Lemma \ref{maximal minimal}.
\end{proof}

\begin{remark}
There is a slight change of notation.
In \cite[(4.10)]{TT}, we put
$G_{\psi_q} \coloneqq \sum_{x \in \mathbb{F}_q}
 \psi_q(x^2)$.
Hence we have $G(\psi_q)=-G_{\psi_q}$.
It eliminates the sign in the formula of $\tau_{\xi}$;
compare with \cite[Corollary A.8]{TT}.
\end{remark}
 
\section{Criteria for maximality and minimality}
\label{Criteria}

In this section, we introduce a condition (Condition \ref{ast})
originally introduced by the second author \cite{Ta}.
We prove criteria for the 
van der Geer--van der Vlugt curves
satisfying Condition \ref{ast}
to be maximal or minimal.

\subsection{Calculation of quadratic Gauss sums}

We use the same notation as in Section \ref{Explicit Formula}.
In particular, we assume $p_0 \neq 2$ and $e \geq 1$,
and fix a maximal abelian subgroup $A \subset H_R$ such that
$A\subset \mathbb{F}_q^2$.
We have $\overline{A} = \pi(A) \subset \mathbb{F}_q$.
We put $q = p^n$.

\begin{lemma}\label{ttl}
 We have 
 \[
  \left(\frac{c_A}{q}\right)=
  \left(\frac{2 a_e}{q}\right). 
 \]
\end{lemma}

\begin{proof} 
 If $\alpha \in \overline{A}$, we have 
 $-\alpha \in \overline{A}$.
 Since $\overline{A}$ is an $\mathbb{F}_p$-vector space of dimension $e$, we obtain 
 \[
 \left(\frac{c_A}{q}\right)=
 \left(\frac{-1}{q}\right)^{e + (p^e-1)/2} \cdot \left(\frac{2 a_e}{q}\right)
 \]
 by \eqref{cc1}. 
 If $p \equiv 1 \pmod 4$, 
we have $\displaystyle \left(\frac{-1}{q}\right)=1$
and the claim follows.
 If $p \equiv 3 \pmod 4$, 
 then $e + (p^e-1)/2$ is even regardless of the parity of $e$ and the claim follows.
\end{proof}

\begin{corollary}\label{tct}
Let the notation and the assumption be 
as in Theorem \ref{tt}. 
 We have 
 \[
 \tau_{\xi}=
 \psi_q(-4^{-1} c_A^{-1} \eta^2)
 \left(\frac{2 a_e}{q}\right) G(\psi_q). 
 \]
\end{corollary}

\begin{proof}
This is a direct consequence of Theorem \ref{tt} (2) and Lemma \ref{ttl}. 
\end{proof}

Recall that $q=p_0^{f_0}$ and $\sqrt{q} = p_0^{f_0/2}$.

\begin{lemma}\label{HG}
Assume that $f_0$ is even. Then we have
\[
G(\psi_q)=(-1)^{f_0 (p_0-1)/4} \sqrt{q}. 
\]
\end{lemma}

\begin{proof}
By \eqref{gauss sum square} and the Hasse--Davenport relation
\cite{Hasse-Dapenport} (see also \cite[Th\'eor\`eme 1.15]{SommesTrig}, \cite[Chapter 11, \S 3, Theorem 1]{Ireland-Rosen}, \cite[Theorem 5.14]{Lidl-Niederreiter}),
we calculate
\begin{align*}
  G(\psi_q) &= G(\psi_{\sqrt{q}})^2=
\left(\frac{-1}{\sqrt{q}}\right)
\sqrt{q}=
\left(\frac{-1}{p_0}\right)^{f_0/2}
\sqrt{q}= (-1)^{f_0 (p_0-1)/4} \sqrt{q}. 
\end{align*}
\end{proof}

Recall that $q = p^n$. 

\begin{lemma}\label{cpql}
Let $k$ be an odd integer.
The value of $G(\psi_q)^k$ is independent of
the choice of a non-trivial character $\psi \in 
\mathbb{F}_p^{\vee} \setminus \{1\}$ if and only if 
$n$ is even. 
\end{lemma}

\begin{proof}
We fix $\psi \in \mathbb{F}_p^{\vee} \setminus \{1\}$. 
For any $\psi' \in \mathbb{F}_p^{\vee}$,  
there exists an element $\la \in \mathbb{F}_p^{\times}$
such that $\psi'(x)=\psi(\la x)$
for every $x \in \mathbb{F}_p$.
This implies 
$\psi'_q(x)=\psi_q(\la x)$ for every $x \in \mathbb{F}_q$.
Hence $\displaystyle G(\psi'_q)= \left( \frac{\la}{q} \right)
G(\psi_q)$.
Since $k$ is odd, we have
\[
  \textrm{$G(\psi_q)^k$ is independent of $\psi \in 
\mathbb{F}_p^{\vee} \setminus \{1\}$}
  \iff
  \textrm{$\left(\frac{\la}{q}\right)=1$ for every $\la \in \mathbb{F}_p^{\times}$}. 
\]
If $n$ is even, we put $n = 2m$.
Then we calculate 
\[
\displaystyle \left( \frac{\la}{q} \right)
= \lambda^{(q-1)/2}
= \lambda^{(p^m-1)(p^m+1)/2}
= (\lambda^{p^m-1})^{(p^m+1)/2}
= 1 
\]
for any $\la \in \mathbb{F}_p^{\times}$. 

To the contrary, assume that
$\displaystyle \left( \frac{\la}{q} \right) = 1$ for every 
$\la \in \mathbb{F}_p^{\times}$. 
If $\la_0$ is a generator of the cyclic group $\mathbb{F}_p^{\times}$, 
we calculate 
\[
1 = \left(\frac{\la_0}{q}\right)
  = \la_0^{(q-1)/2}
  = \la_0^{(p^n-1)/2}
  = \left(\la_0^{(p-1)/2}\right)^{p^{n-1} + \cdots + 1}
  = (-1)^n. 
\]
Thus $n$ is even.
\end{proof}

\subsection{Maximality and minimality over $\mathbb{F}_{q^{p_0}}$}

Combining above results,
we determine when
certain van der Geer--van der Vlugt curves are
$\mathbb{F}_{q^{p_0}}$-maximal
or $\mathbb{F}_{q^{p_0}}$-minimal.

\begin{theorem}\label{cpq}
Assume  
$p_0\neq 2$, $e \geq 1$, and $A\subset \mathbb{F}_q^2$.
Recall that $q=p^n$ and $q=p_0^{f_0}$. 
\begin{itemize}
\item[{\rm (1)}]
Assume that $n$ is even and $p_0 \equiv 1 \pmod 4$. 
The curve 
$\overline{C}_R$ is $\mathbb{F}_{q^{p_0}}$-maximal if $\displaystyle \left( \frac{a_e}{q} \right) = -1$.
It is $\mathbb{F}_{q^{p_0}}$-minimal if $\displaystyle \left( \frac{a_e}{q} \right) = 1$.

\item[{\rm (2)}]
Assume that $n$ is even and 
$p_0 \equiv 3 \pmod 4$.
The curve 
$\overline{C}_R$ is $\mathbb{F}_{q^{p_0}}$-maximal if
$\displaystyle \left( \frac{a_e}{q} \right) = -(-1)^{f_0/2}$.
It is $\mathbb{F}_{q^{p_0}}$-minimal
if $\displaystyle \left( \frac{a_e}{q} \right) = (-1)^{f_0/2}$.

\item[{\rm (3)}]
Let $k \geq 1$ be a positive integer.
If $n$ is odd, the curve $\overline{C}_R$ is neither $\mathbb{F}_{q^k}$-maximal nor $\mathbb{F}_{q^k}$-minimal.
\end{itemize}
\end{theorem}

\begin{proof}
Assume that $n$ is even.
Since $\psi_q$ is a character of order $p_0$, we have $\psi_q^{p_0} = 1$.
By Corollary \ref{tct}, Lemma \ref{HG}, and $p_0 \neq 2$, we have
\[
  \tau_{\xi}^{p_0}
  = \left(\frac{2a_e}{q}\right) 
G(\psi_q)^{p_0}
  = \left(\frac{a_e}{q}\right)(-1)^{f_0(p_0-1)/4} (\sqrt{q})^{p_0}. 
\]
Thus the claim (1) and (2) follow.

The claim (3) follows from Theorem \ref{tt} (2) and Lemma \ref{cpql}. 
\end{proof}

\subsection{Maximality and minimality over $\mathbb{F}_{q}$}

In order to study the $\mathbb{F}_q$-maximality and the $\mathbb{F}_q$-minimality of van der Geer--van der Vlugt curves, we consider the following condition.

Let $a(x) \in \mathbb{F}_q[x]$ be an $\mathbb{F}_p$-linearized polynomial as in \eqref{af}. Recall
$\Ker a=\{x\in \mathbb{F} \mid a(x)=0\}$. 

\begin{condition}[Tatematsu \cite{Ta}]
\label{ast}
For any element  
$\la \in \mathbb{F}_q$ satisfying $\Tr_{q/p}(\la t)=0$
for every $t \in \Ker a$, we have
\[ \Tr_{q/p}(c_A^{-1} \la^2) = 0. \]
\end{condition}

\begin{lemma}\label{tta}
Assume $p_0\neq 2$, $e \geq 1$, and $A\subset \mathbb{F}_q^2$.
The curve $\overline{C}_R$ satisfies
Condition \ref{ast}
if and only if,  
for any $\la \in \mathbb{F}_q$, there 
exists $t \in \Ker a$ such that 
\[
  \Tr_{q/p}(-4^{-1} c_A^{-1} \la^2) = \Tr_{q/p}(\la t). 
\] 
\end{lemma}
\begin{proof}
Assume that the curve $\overline{C}_R$ satisfies
Condition \ref{ast}.
Let $\la \in \mathbb{F}_q$. 
If $\Tr_{q/p}(\la t)=0$ for every $t \in \Ker a$, we have 
$\Tr_{q/p}(-4^{-1} c_A^{-1} \la^2) = \Tr_{q/p}(\la t)=0$ for every $t \in \Ker a$ by Condition \ref{ast}.
Thus we may assume that there exists $t \in \Ker a$ such that $\Tr_{q/p}(\la t) \neq 0$. 
Let 
\[
\mu \coloneqq \frac{\Tr_{q/p}(-4^{-1} c_A^{-1} \la^2)}{
\Tr_{q/p}(\la t). 
}
\]
Then $\mu t \in \Ker a$ by $\mu \in \mathbb{F}_p$, and we have
$\Tr_{q/p}(-4^{-1} c_A^{-1} \la^2) = \Tr_{q/p}(\la \mu t)$.
Hence, the claim follows. 
The `if' part is clear. 
\end{proof}

Recall that for a character $\xi \in A_{\psi}^{\vee}$,
the character $\xi' \in \mathbb{F}_q^{\vee}$
is defined by
$\xi' \coloneqq \xi \circ \phi_{\overline{A}} \circ a$,
and
$\eta \in \mathbb{F}_q^{\times}$ is a unique element
such that $\xi'(x) = \psi_q(\eta x)$
for every $x \in \mathbb{F}_q$;
see Subsection \ref{Heisenberg group}.

\begin{lemma}\label{ttb}
Assume
$p_0\neq 2$, $e \geq 1$, and $A\subset \mathbb{F}_q^2$.
The curve $\overline{C}_R$ satisfies
Condition \ref{ast}
if and only if
\[ \psi_q(-4^{-1} c_A^{-1} \eta^2) = 1 \]
for any $\psi \in \mathbb{F}_p^{\vee}
\setminus \{1\}$ and $\xi \in A_{\psi}^{\vee}$.
\end{lemma}

\begin{proof}

Assume that
$\psi_q(-4^{-1} c_A^{-1} \eta^2) = 1$
for any $\psi \in \mathbb{F}_p^{\vee}
\setminus \{1\}$ and $\xi \in A_{\psi}^{\vee}$.
We check the condition in Lemma \ref{tta}. 
Let $\eta_0 \in \mathbb{F}_q$.
We set 
\[ V_{\eta_0}\coloneqq\{\Tr_{q/p}(\eta_0 t) \mid 
t \in \Ker a\}. \] 
Since $a(x)$ is an $\mathbb{F}_p$-linearized polynomial,
$\Ker a$ is an $\mathbb{F}_p$-vector space.
Thus, $V_{\eta_0} = \mathbb{F}_p$ or $0$.
\begin{itemize}
\item
If $V_{\eta_0}=\mathbb{F}_p$,
there exists $t \in \Ker a$ such that 
\[
\Tr_{q/p}(-4^{-1} c_A^{-1} \eta_0^2) = \Tr_{q/p}(\eta_0 t),
\]
and the condition in Lemma \ref{tta} is satisfied. 
\item
If $V_{\eta_0}=0$, we take any $\psi_0 \in \mathbb{F}_p^{\vee} \setminus \{1\}$.
We define 
\[
\xi_0 \colon \Ima (\phi_{\overline{A}} \circ a)
\to \overline{\mathbb{Q}}_{\ell}^{\times};\ 
\phi_{\overline{A}} \circ a(t) \mapsto 
(\psi_0)_q(\eta_0 t) = \psi_0(\Tr_{q/p}(\eta_0 t)),
\]
which is well-defined by $V_{\eta_0} = 0$.
This $\xi_0$ extends to a character $\xi$ of $A$ such that 
$\xi \in A_{\psi_0}^{\vee}$
since $A \simeq \Ima (\phi_{\overline{A}} \circ a) \times 
\mathbb{F}_p$. 
By the assumption, we have 
$(\psi_0)_q(-4^{-1} c_A^{-1} \eta_0^2) = 1$.
Since $\psi_0 \in \mathbb{F}_p^{\vee} \setminus \{1\}$
is arbitrary,
we have
$\Tr_{q/p}(-4^{-1} c_A^{-1} \eta_0^2) = 0$,
and the condition in Lemma \ref{tta} is satisfied. 
\end{itemize}

To the contrary, assume Condition \ref{ast}.
Let  $\psi \in \mathbb{F}_p^{\vee}
\setminus \{1\}$ and $\xi \in A_{\psi}^{\vee}$. 
From Lemma \ref{tta}, 
there exists $t \in \Ker a$ such that 
\[
\Tr_{q/p}(-4^{-1} c_A^{-1} \eta^2) = \Tr_{q/p}(\eta t)
\]
for some $t \in \Ker a$.
Then we have
$\xi'(t) = (\xi \circ \phi_{\overline{A}} \circ a)(t) = 0$.
Therefore, we have
\[
\psi_q(-4^{-1} c_A^{-1} \eta^2) = \psi_q(\eta t) = \xi'(t) = 1.
\]
Thus the claim follows.
\end{proof}

\begin{theorem}\label{ttbb}
Assume $p_0\neq 2$, $e \geq 1$, and $A\subset \mathbb{F}_q^2$.
Let $k \geq 1$ be a positive integer 
prime to $p_0$.
If $\overline{C}_R$ does not satisfy Condition \ref{ast}, 
the curve $\overline{C}_R$ is 
neither $\mathbb{F}_{q^k}$-maximal nor 
$\mathbb{F}_{q^k}$-minimal. 
\end{theorem}

\begin{proof}
Since $\overline{C}_R$ does not satisfy Condition \ref{ast},
we have $\psi_q(-4^{-1} c_A^{-1} \eta^2) \neq 1$
for some $\psi \in \mathbb{F}_p^{\vee}
\setminus \{1\}$ and $\xi \in A_{\psi}^{\vee}$
by Lemma \ref{ttb}.
The value of $\psi_q(-4^{-1} c_A^{-1} \eta^2)$
is a primitive $p_0$-th root of unity.
Since $k$ is prime to $p_0$,
$\psi_q(-4^{-1} c_A^{-1} \eta^2)^k$
is also a primitive $p_0$-th root of unity.
Hence the curve $\overline{C}_R$ is 
neither $\mathbb{F}_{q^k}$-maximal nor 
$\mathbb{F}_{q^k}$-minimal by Theorem \ref{tt}.
\end{proof}

\begin{theorem}\label{ttb3}
Assume
$p_0\neq 2$, $e \geq 1$, and $A\subset \mathbb{F}_q^2$.
Suppose that the curve $\overline{C}_R$ satisfies Condition \ref{ast}.
Then the following hold.
\begin{itemize}
\item[{\rm (1)}] 
The curve $\overline{C}_R$ is $\mathbb{F}_{q^4}$-minimal. 
\item[{\rm (2)}]
If $q \equiv 1 \pmod 4$, 
the curve $\overline{C}_R$ is $\mathbb{F}_{q^2}$-minimal.
\item[{\rm (3)}]
If $q \equiv 3 \pmod 4$,   
the curve $\overline{C}_R$ is $\mathbb{F}_{q^2}$-maximal. 
\end{itemize}
\end{theorem}
\begin{proof}
By \eqref{gauss sum square}, Theorem \ref{tt}, and Lemma \ref{ttb}, we have
\[
\tau_{\xi}^2
= \psi_q(-4^{-1} c_A^{-1} \eta^2)^2 \cdot G(\psi_q)^2
= \left(\frac{-1}{q}\right) q. 
\]
Thus the claim follows. 
\end{proof}

Recall that $q=p^n$ and $q=p_0^{f_0}$.

\begin{theorem}\label{ttb4}
Assume $p_0\neq 2$, $e \geq 1$, and $A\subset \mathbb{F}_q^2$.
Suppose that the curve $\overline{C}_R$ satisfies Condition \ref{ast}.
\begin{itemize}
\item[{\rm (1)}] 
Assume that $n$ is even and $p_0 \equiv 1 \pmod 4$. 
The curve $\overline{C}_R$ is $\mathbb{F}_{q}$-maximal
if $\displaystyle \left( \frac{a_e}{q} \right) = -1$.
It is $\mathbb{F}_{q}$-minimal if
$\displaystyle \left( \frac{a_e}{q} \right) = 1$. 

\item[{\rm (2)}]
Assume that $n$ is even and $p_0 \equiv 3 \pmod 4$. 
The curve $\overline{C}_R$ is $\mathbb{F}_{q}$-maximal 
if $\displaystyle \left( \frac{a_e}{q} \right) = -(-1)^{f_0/2}$.
It is $\mathbb{F}_{q}$-minimal if
$\displaystyle \left( \frac{a_e}{q} \right) = (-1)^{f_0/2}$.  

\item[{\rm (3)}]
If $n$ is odd, the curve $\overline{C}_R$ is neither  $\mathbb{F}_{q}$-maximal nor $\mathbb{F}_{q}$-minimal. 
\end{itemize}
\end{theorem}

\begin{proof}
Assume that $n$ is even.
Then we have $\displaystyle \left(\frac{2}{q}\right) = 1$.
From Corollary \ref{tct}, Lemma \ref{HG},
and Lemma \ref{ttb}, it follows that 
\[
\tau_{\xi} = \left(\frac{2a_e}{q}\right)G(\psi_q)
= \left(\frac{a_e}{q}\right)
(-1)^{f_0(p_0-1)/4} \sqrt{q}. 
\]
Thus $\overline{C}_R$ is $\mathbb{F}_q$-maximal (resp.\ $\mathbb{F}_q$-minimal) if and only if $\displaystyle \left( \frac{a_e}{q} \right) = -(-1)^{f_0(p_0-1)/4}$ (resp.\ $\displaystyle \left( \frac{a_e}{q} \right)=(-1)^{f_0(p_0-1)/4}$). 
The claims (1) and (2) follow. 

The claim (3) is a special case of Theorem \ref{cpq} (3).
\end{proof}

\subsection{Maximality and minimality of the van der Geer--van der Vlugt curves}

In this subsection,
we collect several results concerning
the maximality and minimality of the van der Geer--van der Vlugt curves.
The results in this subsection are presumably well-known to the specialists.

\begin{proposition}
\label{split}
We write $q=p^n$. 
Assume that $p_0\neq 2$ and $n$ is even.
The following conditions are equivalent.
\begin{itemize}
\item[{\rm (1)}]
The curve $\overline{C}_R$ is $\mathbb{F}_{q}$-maximal or $\mathbb{F}_{q}$-minimal. 

\item[{\rm (2)}]
We have $V_R \subset \mathbb{F}_q$.
Namely, $\mathbb{F}_q$ is a splitting field of $E_R(x)$. 
\end{itemize}
\end{proposition}

\begin{proof}
When $p$ is a prime number (i.e.,\ $p = p_0$),
it follows from \cite[Proposition 13.4]{GV}.
Here we give a proof using the results of
\c Cak\c cak--\"Ozbudak \cite{CO}.
By \cite[Corollary 3.2]{CO},
the curve $\overline{C}_R$ is $\mathbb{F}_{q}$-maximal or $\mathbb{F}_{q}$-minimal
if and only if ``$k = 2h$'' is satisfied in the notation of \cite{CO}.
Here, the value of ``$k$'' in \cite{CO} is equal to
$\dim_{\mathbb{F}_p} (V_R \cap \mathbb{F}_q)$
and the value of ``$h$'' in \cite{CO}
is equal to $e$.
Therefore, the condition ``$k = 2h$'' in \cite{CO}
amounts to $\dim_{\mathbb{F}_p} (V_R \cap \mathbb{F}_q) = 2e$.
This is equivalent to $V_R \subset \mathbb{F}_q$
since $V_R$ is an $\mathbb{F}_p$-vector space of dimension $2e$.
\end{proof}

\begin{proposition}
Assume $p_0\neq 2$.
Let $f \geq 1$ be a positive integer such that $\mathbb{F}_q \subset \mathbb{F}_{p^{2f}}$.
If the curve $\overline{C}_R$ is 
$\mathbb{F}_{p^{2f}}$-maximal, then we have $e \leq f$. 
\end{proposition}

\begin{proof}
By Proposition \ref{split},
$V_R$ is an $\mathbb{F}_p$-vector space of dimension $2e$.
Since $V_R \subset \mathbb{F}_q$,
we have $2e \leq [\mathbb{F}_q : \mathbb{F}_p] \leq 2f$.
Thus $e \leq f$. 
\end{proof}

In the following proposition, we do not assume $p_0 \neq 2$.

\begin{proposition}\label{pp}
Assume $e =f$ and $q = p^{2f}$.
Namely, let
\[
  R(x)=\sum_{i=0}^f a_i x^{p^i} \in 
\mathbb{F}_{p^{2f}}[x]
  \qquad (a_f \neq 0)
\] 
be an $\mathbb{F}_p$-linearized polynomial of degree $p^f$ over $\mathbb{F}_{2f}$.
The following conditions are equivalent. 
\begin{itemize}
\item[{\rm (1)}] 
The curve $\overline{C}_{R}$ is $\mathbb{F}_{p^{2f}}$-maximal.
\item[{\rm (2)}] 
We have $\Tr_{p^{2f}/p}(xR(x))=0$ for every 
$x \in \mathbb{F}_{p^{2f}}$.
\item[{\rm (3)}] We have 
$a_0=\cdots=a_{f-1}=0$ and 
$a_f^{p^f}+a_f=0$.
\end{itemize}
\end{proposition}

\begin{proof}
This result is known (cf.\ \cite[Proposition 2.1]{OS}).
Here we give a different proof.

First we show the equivalence (1) $\iff$ (2).
Since $g(\overline{C}_R)=p^f(p-1)/2$, 
the curve $\overline{C}_R$ is 
$\mathbb{F}_{p^{2f}}$-maximal if and only if 
\[ |\overline{C}_R(\mathbb{F}_{p^{2f}})| = p^{2f} + 1 + p^{2f}(p-1) = p^{2f+1} + 1.
\]
Since the complement $\overline{C}_R \backslash C_R$ consists of
an $\mathbb{F}_{p^{2f}}$-rational point,
this condition is satisfied if and only if, 
for every $x \in \mathbb{F}_{p^{2f}}$, 
there exists a solution
$z \in \mathbb{F}_{p^{2f}}$ of  
$z^p-z=xR(x)$.
The latter condition holds if and only if
$\Tr_{p^{2f}/p}(xR(x))=0$ for every $x \in \mathbb{F}_{p^{2f}}$.
Hence the equivalence (1) $\iff$ (2) follows.

Next, we show (2) $\Rightarrow$ (3).
Assume that (2) holds.
First we show $V_R=\mathbb{F}_{p^{2f}}$.
For every $x,y\in \mathbb{F}_{p^{2f}}$, we have
\begin{align*}
& \Tr_{p^{2f}/p}(xR(y)+y R(x)) \\
&= \Tr_{p^{2f}/p}((x+y)R(x+y))-\Tr_{p^{2f}/p}(xR(x)) -\Tr_{p^{2f}/p}(yR(y)) \\
&= 0. 
\end{align*}
Therefore, from (\ref{11}), we obtain 
\begin{align*}
\Tr_{p^{2f}/p}(x^{p^f} E_R(y))
&= -\Tr_{p^{2f}/p}(f_R(x,y)^p-f_R(x,y))+
\Tr_{p^{2f}/p}(xR(y)+y R(x)) \\
&= 0
\end{align*}
for every $x,y\in \mathbb{F}_{p^{2f}}$.
Since the $\mathbb{F}_p$-bilinear form
\[
T_{p^{2f}/p} \colon \mathbb{F}_{p^{2f}} \times \mathbb{F}_{p^{2f}} \to \mathbb{F}_p;\ (x,y) \mapsto \Tr_{p^{2f}/p}(xy)
\]
is non-degenerate,
we obtain $E_R(y)=0$ for every $y\in \mathbb{F}_{p^{2f}}$.
Hence $\mathbb{F}_{p^{2f}} \subset V_R$.
Since the both sides have the same 
cardinality, we obtain $\mathbb{F}_{p^{2f}}= V_R$. 
Since $E_R(x)$ is a polynomial of degree $p^{2f}$,
we have $E_R(x)=a_f^{p^f}(x^{p^{2f}}-x)$.
From \eqref{er}, we have
\[
E_R(x)=\sum_{i=1}^f a_i^{p^f} x^{p^{f+i}}+2(a_0x)^{p^f}
+\sum_{i=1}^f (a_i x)^{p^{f-i}}. 
\]
Comparing the right hand sides, we obtain (3) when $p_0 \neq 2$.
When $p_0 = 2$, we obtain
$a_1=\cdots=a_{f-1}=0$ and $a_f^{p^f}+a_f=0$. 
It remains to show $a_0=0$. 
For any $x \in \mathbb{F}_{p^{2f}}$, 
we compute 
\begin{align*}
0 &= \Tr_{p^{2f}/p}(xR(x))
= \Tr_{p^{2f}/p}\left(a_fx^{p^f+1}+a_0x^2\right) \\
&=
\Tr_{p^{f}/p}\left( \Tr_{p^{2f}/p^f} \left( a_fx^{p^f+1} \right) \right) +
\Tr_{p^{2f}/p}\left(a_0x^2\right)\\
&=
\Tr_{p^{f}/p}\left(a_f^{p^f}x^{p^f(p^f+1)}+a_f x^{p^f+1}\right)
+\Tr_{p^{2f}/p}(a_0 x^2) \\
&=\Tr_{p^{2f}/p}(a_0 x^2). 
\end{align*}
When $p_0 = 2$, the map 
$\mathbb{F}_{p^{2f}} \to \mathbb{F}_{p^{2f}}; x \mapsto x^2$
is bijective,
and the non-degeneracy of $T_{p^{2f}/p}$
implies $a_0=0$.
Thus we obtain (3) in all characteristics.  

Finally, we show (3) $\Rightarrow$ (2).
Assume that (3) holds.
For every $x \in \mathbb{F}_{p^{2f}}$,
we have
\begin{align*}
\Tr_{p^{2f}/p}(xR(x)) &=
\Tr_{p^{f}/p}\left( \Tr_{p^{2f}/p^f}\left( a_f x^{p^f + 1} \right) \right)
 = \Tr_{p^{f}/p}\left(a_f^{p^f}x^{p^f(p^f+1)}+a_f x^{p^f+1}\right)
 \\
 &= \Tr_{p^{f}/p}\left(-a_f x^{p^f+1}+a_f x^{p^f+1}\right)
  = 0.
\end{align*}
Thus (2) follows.
\end{proof}

\section{Some examples of maximal van der Geer--van der Vlugt curves}
\label{Examples maximal van der Geer--van der Vlugt curves}

In general, Condition \ref{ast} is not easy to check.
Here we give a sufficient condition
which implies Condition \ref{ast}; see Lemma \ref{tt2}.
Applying this lemma, we give explicit
maximal (or minimal) curves in Theorem \ref{214}.

\begin{lemma}\label{tt1}
Assume $p_0 \neq 2$, $q = p^n$, and $p \equiv 1 \pmod n$.
Let $R(x)=\sum_{i=0}^e a_i x^{p^i} \in \mathbb{F}_p[x]$
be an $\mathbb{F}_p$-linearized polynomial of degree $p^e$ over $\mathbb{F}_p$
with $a_e \neq 0$.
We take a primitive 
$n$-th root of unity 
$\zeta \in \mathbb{F}_p$ and 
an element $\alpha \in 
\mathbb{F}_q^{\times}$ with 
$\alpha^{p-1} = \zeta$. 
If integers $k$ and $l$ satisfy 
\[
E_R(\alpha^k)=E_R(\alpha^l)=0,  
\qquad k+l \not\equiv 0 \pmod{n}, 
\]
then we have 
\[
f_R(\alpha^k,\alpha^l)=f_R(\alpha^l,\alpha^k). 
\]
\end{lemma}

\begin{proof}
Since $a_i \in \mathbb{F}_p$, we have $a_i^p = a_i$.
We have $\alpha^{p^i}=\zeta^i \alpha$ for 
an integer $i$. 
By \eqref{fr} and $k+l \not\equiv 0 \pmod{n}$, we have 
\begin{align*}
f_R(\alpha^k,\alpha^l)\alpha^{-(k+l)}
&=-\sum_{i=0}^{e-1}
\left(\sum_{j=0}^{e-i-1}a_i \zeta^{(i+j)k+jl}
+\sum_{h=0}^e a_h \zeta^{ik + (h+i)l}\right) \\
&=
-\sum_{i=0}^{e-1} a_i\zeta^{ik}
\frac{1-\zeta^{(e-i)(k+l)}}{1-\zeta^{k+l}}
-\sum_{h=0}^e a_h\zeta^{hl}
\frac{1-\zeta^{e(k+l)}}{1-\zeta^{k+l}} \\
&=
-\frac{1}{1-\zeta^{k+l}}\left(\sum_{i=0}^{e} a_i(\zeta^{ik}+\zeta^{il})
-\zeta^{e(k+l)} \sum_{i=0}^e a_i
(\zeta^{il}+\zeta^{-il})\right). 
\end{align*}
Since $E_R(\alpha^l) = 0$, we have
\[
0=E_R(\alpha^l)\alpha^{-l}=\zeta^{el}
\sum_{i=0}^e a_i
(\zeta^{il}+\zeta^{-il}) 
\]
by \eqref{er}.
Hence, we have
\[
f_R(\alpha^k,\alpha^l)\alpha^{-(k+l)}
= -\frac{1}{1-\zeta^{k+l}} \sum_{i=0}^{e} a_i(\zeta^{ik}+\zeta^{il}).
\]
Since the right hand side is symmetric in $k$ and $l$, the claim follows.
\end{proof}

\begin{lemma}\label{tt20}
Assume $p_0 \neq 2$, $q = p^n$, $p \equiv 1 \pmod n$
and $e \geq 1$. 
Let $\zeta \in \mathbb{F}_p$ 
be a primitive $n$-th root of unity. 
We take an element 
$\alpha \in \mathbb{F}_q^{\times}$ with
$\alpha^{p-1} = \zeta$. 
We define 
\[
f_k(x) \coloneqq x^p-\zeta^k x \in \mathbb{F}_p[x]. 
\]
Assume that there exist different integers 
$k_1,\ldots,k_e \in \{0,1,\ldots,n-1\}$
such that 
\[
F_A(x)
=f_{k_1} \circ \cdots \circ f_{k_e}(x). 
\]
\begin{itemize}
\item[{\rm (1)}] The $\mathbb{F}_p$-vector space 
$\mathbb{F}_q$ has a basis 
$\{1,\alpha,\ldots,\alpha^{n-1}\}$. 
\item[{\rm (2)}] 
We have
$\overline{A} = \bigoplus_{i=1}^{e}\mathbb{F}_{p} \alpha^{k_i} \subset
\mathbb{F}_q$ and $A \subset \mathbb{F}_q^2$. 
Let
\[ \mathcal{L} \coloneqq \{0,1,\ldots,n-1\} \setminus 
\{k_1,\ldots,k_e\}. \]
Let $a(x)$ be as in \eqref{af}. 
Then, 
$\Ker a$ is the $\mathbb{F}_p$-vector space 
with a basis 
$\{\alpha^l \mid l \in \mathcal{L}\}$. 
\end{itemize}
\end{lemma}
\begin{proof}
(1) It suffices to show that 
 $\{1,\alpha,\ldots,\alpha^{n-1}\}$ is
an $\mathbb{F}_p$-basis of $\mathbb{F}_q$.
 Assume 
 \[
 \sum_{i=0}^{n-1} c_i \alpha^i=0
 \qquad (c_i \in \mathbb{F}_{p}). 
 \]
 Taking the $p^{j}$-th power for $0 \leq j \leq n-1$, we have
 \[
 \sum_{i=0}^{n-1} \zeta^{ij} c_i \alpha^i=0 \qquad 
 (0 \leq j \leq n-1). 
 \]
 Since the determinant of the Vandermonde matrix 
 $(\zeta^{ij})_{0 \leq i,j \leq n-1}$
 is $\prod_{0 \leq i<j \leq n-1}
 (\zeta^j-\zeta^i) \neq 0$, we obtain 
 $c_i \alpha^i = 0$ for every $i$. 
Hence $c_i = 0$ for every $i$. 

(2) 
Since 
the polynomial $f_k(x)$ is $\mathbb{F}_p$-linearized
for any positive integer $k$, so is $F_A(x)$. 
Hence $\overline{A}=\Ker F_A$ is an $\mathbb{F}_{p}$-vector space. 
Since 
$f_{k_i} \circ f_{k_j}=f_{k_j} \circ f_{k_i}$
for any $1 \leq i < j \leq e$, 
we have 
 $\alpha^{k_1},\ldots,
 \alpha^{k_e} \in \overline{A}$. 
 From (1), it follows that
 \[
   \bigoplus_{i=1}^{e}\mathbb{F}_{p} \alpha^{k_i} \subset \overline{A}.
 \]
By comparing the orders of the both sides,
this inclusion is equal. 
In particular, $\overline{A} \subset 
\mathbb{F}_q$ since $\alpha \in \mathbb{F}_q$. 
According to \cite[Lemma 2.6]{TT}, we obtain 
$A \subset \mathbb{F}_q^2$. 
 
 We show the latter claim. 
We write $\mathcal{L}=\{l_1,\ldots,l_{n-e}\}$. 
Let 
\[
  b(x) \coloneqq f_{l_1} \circ \cdots \circ f_{l_{n-e}}(x).
\]
We show $a(x)=b(x)$. 
Let $f(x) \coloneqq F_A \circ b(x)$.
Since $f(x)$ is $\mathbb{F}_p$-linearized,  
$\Ker f$ is an $\mathbb{F}_{p}$-vector space. 
We have $f(\alpha^i)=0$ for every integer 
$0 \leq i \leq n-1$.  
By (1), we have 
$\mathbb{F}_q=\Ker f$. Since 
$f(x)$ is a monic polynomial of degree 
$q$, we have $f(x) = x^q - x$. 
Thus, $F_A \circ a(x)=F_A \circ b(x)$ by \eqref{af}.
Since $F_A(x)$ is an additive polynomial, we have
$F_A(a(x) - b(x)) = 0$.
Hence $a(x)=b(x)$ by $F_A(x) \neq 0$. 
%Since the ring of all additive polynomials is an integral domain, we obtain 

Since $a(x)$ is $\mathbb{F}_p$-linearized, $\Ker a$ is an 
$\mathbb{F}_{p}$-vector space. 
Clearly $\alpha^l \in \Ker b=\Ker a$ for any 
$l \in \mathcal{L}$. 
By considering the degree of $a(x)$, we obtain
\[
  \bigoplus_{l \in 
\mathcal{L}} \mathbb{F}_{p} \alpha^l = \Ker a,
\]
i.e., $\{\alpha^l \mid l \in \mathcal{L}\}$ is an $\mathbb{F}_{p}$-basis of $\Ker a$.
\end{proof}

Recall that $R(x)=\sum_{i=0}^e a_i x^{p^i} \in \mathbb{F}_q[x]$ is an $\mathbb{F}_p$-linearized polynomial of degree $p^e$, and $a_e \in \mathbb{F}_q^{\times}$ is the coefficient of $x^{p^e}$ in $R(x)$.

\begin{lemma}\label{tt2}
Let the notation and the assumptions be as in Lemma \ref{tt20}. Further, we assume 
$a_e \in \mathbb{F}_p$. 
We put 
\[
\mathcal{M} \coloneqq \{ k_i+k_j \mid 1 \leq i,j  \leq e\}.
\]
The set $\mathcal{M}$ does not contain $0$ nor 
$n$ 
if and only if Condition \ref{ast} is satisfied.
\end{lemma}

\begin{proof}
We write $F_A(x)=\sum_{i=0}^e b_i x^{p^i}$. 
We have $b_e=1$ and 
$b_0=\prod_{i=1}^e (-\zeta^{k_i}) \in \mathbb{F}_p$
by $\zeta \in \mathbb{F}_p$. 
From \eqref{cc1} and the assumption 
$a_e \in \mathbb{F}_p$, 
it follows that 
\begin{equation}\label{ca}
c_A \in \mathbb{F}_p. 
\end{equation}

We recall that $\alpha^{p^i}=\zeta^i \alpha$ for 
an integer $i$. 
For $s \in \mathbb{F}_p$ and 
an integer $k$ which is not divisible by $n$, we have
\begin{equation}\label{ket}
\Tr_{q/p}(s\alpha^k)=
s\sum_{i=0}^{n-1} (\alpha^k)^{p^i} = 
s \alpha^k\sum_{i=0}^{n-1} \zeta^{ik}=
s \alpha^k \frac{1-\zeta^{nk}}{1-\zeta^k}=0
\end{equation}
since $\zeta^n = 1$.

Let $\mathcal{L}$ be as in Lemma \ref{tt20} (2).
Take an element $\la=\sum_{i=0}^{n-1} \la_i \alpha^i \in \mathbb{F}_q$
with $\la_i \in \mathbb{F}_{p}$.
We set $\la_n \coloneqq \la_0$.
First, we prove the following conditions are equivalent:
\begin{itemize}
\item[{\rm (a)}] $\Tr_{q/p}(\la t)=0$ for any $t \in \Ker a$. 
\item[{\rm (b)}] $\la_{n-l}=0$ for any $l \in \mathcal{L}$.
\end{itemize}

We note $(\alpha^n)^{p}
=(\alpha^{p})^n=(\zeta \alpha)^n=\alpha^n$.
Thus $\alpha^n \in \mathbb{F}_p^{\times}$.
Using \eqref{ket}, 
for any $s \in \mathbb{F}_{p}$
and $l \in \mathcal{L}$, we calculate 
\begin{align*}
\Tr_{q/p}(\la s \alpha^l)&=
\begin{cases}
\Tr_{q/p}(\la_{n-l} s \alpha^n) & 
\textrm{if $l \neq 0$}, \\
\Tr_{q/p}(\la_0 s) & 
\textrm{if $l = 0$}
\end{cases} \\
&=
\begin{cases}
n \la_{n-l} s \alpha^n & 
\textrm{if $l \neq 0$}, \\
n \la_0 s & 
\textrm{if $l = 0$}. 
\end{cases}
\end{align*}
By $p \equiv 1 \pmod{n}$, the integer $n$ is not divisible by $p_0$.
Therefore, if $\la$ satisfies the condition (a),
then $\la_{n-l}=0$ for any $l \in \mathcal{L}$
by Lemma \ref{tt20} (2).
To the contrary, if 
$\la_{n-l}=0$ for any 
$l \in \mathcal{L}$,
then $\la$ satisfies the condition (a)
by the above argument.
Hence the conditions (a), (b) are equivalent.

We shall prove the claim of the lemma.
Let $\la \in \mathbb{F}_q$ be an element
satisfying the conditions (a), (b).
Then we can write
\[
\la=\sum_{i=1}^{e} \la_{n-k_i} \alpha^{n-k_i}
\]
for some $\la_{n-k_i} \in \mathbb{F}_{p}$.
Recall that $0 \leq k_i \leq n-1$ for $1 \leq i \leq e$. 
By using \eqref{ca}, \eqref{ket}, and 
$\alpha^n \in \mathbb{F}_p$, we calculate  
\begin{gather}\label{trpq}
        \begin{aligned}
            \Tr_{q/p} (c_A^{-1} \lambda^2) 
            &= c_A^{-1}
                \Tr_{q/p}
                ( \lambda^2  ) \\
            &= c_A^{-1}
                \Tr_{q/p}
                \left(\sum_{\substack{1 \leq i,j \leq e \\ k_i + k_j = 0,\, n}} \la_{n - k_i} \la_{n - k_j}  
                \alpha^{2n - k_i - k_j }\right) \\
            &= c_A^{-1} n 
                 \sum_{\substack{1 \leq i,j \leq e \\ k_i + k_j = 0,\, n}} \la_{n - k_i} \la_{n - k_j} 
                \alpha^{2n - k_i - k_j }
              . 
        \end{aligned}
        \end{gather}
        Thus, if $\mathcal{M}$ does not contain $0$ nor $n$, the above sum is zero. Hence Condition \ref{ast} is satisfied.

To the contrary, assume that Condition \ref{ast} is satisfied.
Let $\la \in \mathbb{F}_q$ be an element
satisfying the conditions (a), (b).
We put 
\[
T({\la}) \coloneqq \sum_{\substack{1 \leq i,j \leq e \\ k_i + k_j = 0,\, n}} \la_{n - k_i} \la_{n - k_j}  
                \alpha^{2n - k_i - k_j}. 
\]
By \eqref{trpq}, we have  
$T({\la})=0$. 

We show that $\mathcal{M}$ does not contain $0$ nor $n$
by contradiction.
\begin{itemize}
\item Assume that there exists an integer $1 \leq i \leq e$ such that $k_i=0$.
Then $\la = 1$ satisfies the conditions (a), (b).
Taking $\la=1$, we have $T(1) = 1 \neq 0$, which is a contradiction. 

\item Assume that $n$ is even and there exists an integer $1 \leq i \leq e$ such that 
$k_i=n/2$.
Then $\la = \alpha^{n/2}$ satisfies the conditions (a), (b).
Taking $\la=\alpha^{n/2}$, 
we have $T(\alpha^{n/2}) = \alpha^n \neq 0$, which is a contradiction. 

\item Assume that there exist integers $1 \leq i,j \leq e$ such that 
$k_i \neq k_j$, $k_i \notin \{ 0,\,n/2 \}$, $k_j \notin \{ 0,\,n/2 \}$, and $k_i+k_j=n$. 
Then $\la = \alpha^{n-k_i}+\alpha^{n-k_j}$ satisfies the conditions (a), (b).
Taking $\la = \alpha^{n-k_i}+\alpha^{n-k_j}$,
we have
$T(\alpha^{n-k_i}+\alpha^{n-k_j}) = 2 \alpha^n \neq 0$, which is a contradiction.
\end{itemize}
Therefore, $\mathcal{M}$ does not contain $0$ nor $n$. 
\end{proof}

In Theorem \ref{214}, we give some examples of van der Geer--van der Vlugt curves to which our criteria can be applied.
These curves were also studied by Bartoli--Quoos--Sayg\i--Y\i lmaz in \cite{B}; see Remark \ref{bqsy}.

\begin{theorem}\label{214}
Assume that 
$p_0 \neq 2$, 
$q=p^n$, and $p \equiv 1 \pmod n$.
Let $e \geq 1$, and $c_0$, $c_1$, $\ldots$, $c_{e-1} \in \mathbb{F}_p$ be elements such that the polynomial 
 \[
        g(x) \coloneqq x^{2e} + c_{e - 1} x^{2 e - 1} + \cdots +  c_1 x^{e + 1} + c_0 x^e + c_1 x^{e - 1}
            + \cdots + c_{e - 1} x + 1 \in \mathbb{F}_p[x]
            \]
            divides $x^n-1$. 
            Let 
              \[
        R(x) \coloneqq 2 x^{p^{e}} + 2 c_{e - 1} x^{p^{e - 1}} + \cdots + 2 c_1 x^p + c_0 x
        \in \mathbb{F}_p[x]. 
    \]
    \begin{itemize}
\item[{\rm (1)}] 
The curve $\overline{C}_R$ is
$\mathbb{F}_{q^2}$-maximal if $q \equiv 3 \pmod 4$.
It is $\mathbb{F}_{q^2}$-minimal 
if $q \equiv 1 \pmod 4$.

\item[{\rm (2)}]
   We write $q = p_0^{f_0}$.
  \begin{itemize}
\item[{\rm (i)}]   If $n$ is even and 
   $p_0 \equiv 1 \pmod 4$, the curve 
   $\overline{C}_R$ is 
   $\mathbb{F}_q$-minimal.
\item[{\rm (ii)}] 
     If $n$ is even, $p_0 \equiv 3 \pmod 4$ and 
    $f_0 \equiv 2 \pmod 4$, the curve 
   $\overline{C}_R$ is 
$\mathbb{F}_q$-maximal.
   If $n$ is even, $p_0 \equiv 3 \pmod 4$ and  
    $f_0 \equiv 0 \pmod 4$, 
    the curve 
   $\overline{C}_R$ is 
    $\mathbb{F}_q$-minimal. 
 \item[{\rm (iii)}]   If $n$ is odd, the curve 
   $\overline{C}_R$ is neither  
$\mathbb{F}_q$-maximal nor $\mathbb{F}_q$-minimal.
\end{itemize}
\end{itemize}
\end{theorem}

\begin{proof}
Let $\kappa \in \{\pm 1\}$. 
Since $g(x)$ divides $x^n-1$ and
$p_0$ does not divide $n$,
the polynomial $g(x)$ is separable.
We can directly check
$g'(\kappa)=e \kappa  g(\kappa)$, where $g'(x)$ is the 
derivative of $g(x)$. 
Thus $g(\kappa) \neq 0$. 

Let $\zeta \in \mathbb{F}_p$ be a primitive 
    $n$-th root of unity. 
%We note that $\alpha^n \in \mathbb{F}_p$ by $\zeta^n=1$. 
Since $g(x)$ is self-reciprocal,
$g(x)$ divides $x^n-1$,
$g(1) \neq 0$, and $g(-1) \neq 0$,
the roots of $g(x)$ have the following form: 
\[
\zeta^{k_1}, \ldots, \zeta^{k_e}, \zeta^{n - k_e}, \ldots, \zeta^{n - k_1} \qquad 
(0< k_1 <  \cdots < k_e < n/2).
\]
By \eqref{er}, we have 
   \[
        E_R(x) = 2 \left(x^{p^{2e}} + c_{e - 1} x^{p^{2 e - 1}} + \cdots +  c_1 x^{p^{e + 1}} + c_0 x^{p^e} + c_1 x^{p^{e - 1}}
            + \cdots + c_{e - 1} x^p + x\right). 
    \]

By $q=p^n$ and $\zeta^n=1$,
there exists an element
$\alpha \in \mathbb{F}_q$
with $\alpha^{p-1} = \zeta$.
By Lemma \ref{tt20} (1), 
it follows that $\{1, \alpha, \ldots,\alpha^{n-1}\}$ is an $\mathbb{F}_p$-basis of $\mathbb{F}_q$.

Since $(\alpha^i)^{p^j} = \zeta^{ij} \alpha^i$,
we have $E_R(\alpha^i)=2g(\zeta^i) \alpha^i$ for every integer $i$.
Hence, we obtain  
\[
E_R(\alpha^{k_i})=0, \quad 
E_R(\alpha^{-k_i})=0
\quad \textrm{for $1 \leq i \leq e$}.  
\]

Letting  
$\overline{A}\coloneqq \bigoplus_{i=1}^e \mathbb{F}_p \alpha^{k_i}$, the inverse image $A=\pi^{-1}(\overline{A})$ is a maximal abelian subgroup of 
$H_R$ by $k_i+k_j \not\equiv 0 \pmod{n}$, \eqref{op}, and Lemma \ref{tt1}. 
By setting 
\[
 F_A(x) \coloneqq f_{k_1} \circ \cdots \circ f_{k_e}(x) \quad \textrm{with 
$f_k(x) \coloneqq x^p-\zeta^k x \in \mathbb{F}_p[x]$}, 
\]
we have 
$\overline{A}=\{x \in \mathbb{F} \mid F_A(x)=0\}$
by Lemma \ref{tt20} (2).
Since $0<k_i+k_j<n$ for any $1 \leq i,j \leq e$, 
the curve $\overline{C}_R$ satisfies Condition \ref{ast} by Lemma \ref{tt2}. 

From Lemma \ref{tt20} (2), we obtain $A \subset \mathbb{F}_q^2$. 
Hence the claim (1) follows from Theorem \ref{ttb3}.
Since $a_e = 2$, we have
$\displaystyle \left(\frac{a_e}{q}\right) = 
 \left(\frac{2}{q}\right) = 1$
when $n$ is even.
Thus, the claim (2) follows from Theorem \ref{ttb4}.
\end{proof}

\begin{example}\label{2R}
Let
\[
R_{+}(x) \coloneqq 2
\sum_{i=1}^e x^{p^i}+x, \qquad
R_{-}(x) \coloneqq 2
\sum_{i=1}^e (-1)^i x^{p^i}+x.
\]
\begin{itemize}
\item[{\rm (1)}] 
We take $\xi \in \mathbb{F}_{p^2}^{\times}$
such that $\xi^p+\xi=0$. 
Then we have the $\mathbb{F}_{p^2}$-isomorphism
$C_{R_+} \xrightarrow{\sim} C_{R_-};\ (x,y) \mapsto 
(\xi x,\xi^2 y)$. 
\item[{\rm (2)}] 
Assume that $p \equiv 3 \pmod 4$, $q=p^{2e+1}$ and 
$p\equiv 1 \pmod{2e+1}$.
We apply Theorem \ref{214} to the case where 
$n=2e+1$ and $g(x)=\sum_{i=0}^{2e} x^i$, which divides $x^n-1$. 
By Theorem \ref{214} (1), the curve $\overline{C}_{R_+}$ is 
$\mathbb{F}_{q^2}$-maximal. 
Thus, both $\overline{C}_{R_{+}}$ and $\overline{C}_{R_{-}}$
are $\mathbb{F}_{q^2}$-maximal by (1). 
\end{itemize}
\end{example}

\section{A generalization of van der Geer--van der Vlugt curves}
\label{Generalization}

Our main aim in this section is to study a curve $\overline{C}_{R,r}$
defined by the equation of the form
\[
z^{p^r}-z= xR(x)
\]
for some $r \geq 1$.
When $r = 1$, the curve $\overline{C}_{R,r}$ is nothing but the van der Geer--van der Vlugt curve
associated with $R(x)$.

We show that $\overline{C}_{R,r}$ is supersingular and 
give a certain criterion for $\overline{C}_{R,r}$
to be maximal in Lemma \ref{fl}. 
We give a necessary and sufficient condition for 
 $\overline{C}_{R,r}$ to be  
$\mathbb{F}_{p^{2e}}$-maximal in Proposition \ref{c1}, 
which is a generalization of 
\"Ozbudak--Sayg\i's result \cite[Proposition 2.1]{OS}. 
Based on this, we propose a conjecture on maximality of $\overline{C}_{R,r}$; see Conjecture \ref{conj}.

\subsection{$L$-polynomials and the supersingularity of $\overline{C}_{R,r}$}

Let $p_0$ be a prime number, and $p$ a power of $p_0$.
In this section, we do not assume $p_0 \neq 2$.

Let $e \geq 0$ and $f \geq 1$ be integers, and
\[
  R(x) = \sum_{i=0}^e a_i x^{p^i} \in \mathbb{F}_{p^{2f}}[x]
  \qquad (a_e \neq 0)
\]
be an $\mathbb{F}_p$-linearized polynomial of degree $p^e$.
For a positive integer $r \geq 1$,
let $C_{R,r}$ be the affine curve over $\mathbb{F}_{p^{2f}}$
defined by 
\[
z^{p^r}-z= xR(x). 
\]
Let $\overline{C}_{R,r}$ be the smooth compactification of $C_{R,r}$.

For $\zeta \in \mathbb{F}^{\times}$,
we put
$\zeta R(x) \coloneqq \sum_{i=0}^e \zeta a_i x^{p^i}$.
For a character
$\psi \in \mathbb{F}_p^{\vee}$, let $\mathscr{L}_{\psi}$ denote the 
$\overline{\mathbb{\mathbb{Q}}}_{\ell}$-sheaf on 
$\mathbb{A}^1 \coloneqq \Spec \mathbb{F}_p[x]$ defined by $\psi$ and the 
Artin--Schreier covering
$\Spec \mathbb{F}_p[x,y]/(x - (y^p - y)) \to \mathbb{A}^1$.
For a morphism between $\mathbb{F}_p$-schemes $X \to \mathbb{A}^1$, 
let $\mathscr{L}_{\psi}(f)$ denote the 
pull-back $f^\ast \mathscr{L}_{\psi}$ to $X$.

\begin{lemma}\label{ail}
If the curve $\overline{C}_{R,r}$ is $\mathbb{F}_{p^{2f}}$-maximal, then $2f \equiv 0 \pmod{r}$. 
\end{lemma}
\begin{proof}
We take $\psi \in\mathbb{F}_p^{\vee} \setminus 
\{1\}$. 
We have isomorphisms 
\begin{equation}\label{ai}
H^1(\overline{C}_{R,r}) \xleftarrow{\sim}
H_{\rm c}^1(C_{R,r}) \simeq 
\bigoplus_{\zeta \in \mathbb{F}_{p^r}^{\times}} 
H_{\rm c}^1(\mathbb{A}^1,
\mathscr{L}_{\psi}(x \zeta R(x))), 
\end{equation}
which are shown in the same way as 
\cite[Lemma 2.8 and Corollary 3.6]{TT}. 
Since 
$\overline{C}_{R,r}$ is $\mathbb{F}_{p^{2f}}$-maximal, 
the subspaces 
$\{H_{\rm c}^1(\mathbb{A}^1,
\mathscr{L}_{\psi}(x \zeta R(x)))\}_{\zeta \in 
\mathbb{F}_{p^r}^{\times}}$ in $H^1(\overline{C}_{R,r})$ via \eqref{ai}
are stable under the action of 
$\mathrm{Fr}_q^\ast$.
Hence we have $\mathbb{F}_{p^r} \subset \mathbb{F}_{p^{2f}}$. 
Thus the claim follows. 
\end{proof}

\begin{corollary}\label{aic}
Assume $R(x) \in \mathbb{F}_{p_0}[x]$. 
If the curve $\overline{C}_{R,r}$ is not $\mathbb{F}_{p^k}$-maximal for any $k \geq 1$,
then $\overline{C}_{R,r}$ is not $\mathbb{F}_{p_0^k}$-maximal
for any $k \geq 1$. 
\end{corollary}

\begin{proof}
Assume that $\overline{C}_{R,r}$ is $\mathbb{F}_{p_0^k}$-maximal for some $k \geq 1$.
By the same argument as the proof of Lemma \ref{ail}, 
we have $\mathbb{F}_{p^r} \subset 
\mathbb{F}_{p_0^k}$.
Hence we can write $p_0^k=p^{k'}$ for some $k'$.
Thus $\overline{C}_{R,r}$ is $\mathbb{F}_{p^{k'}}$-maximal. 
\end{proof}

For an element $\zeta \in \mathbb{F}_{p^r}^{\times}$, 
let $A_{\zeta} \subset H_{\zeta R}$ be a 
maximal abelian subgroup. 
We identify the center of $H_{\zeta R}$
with $\mathbb{F}_p$ as in Section \ref{Explicit Formula}.
We take a character 
$\psi \in \mathbb{F}_p^{\vee} \setminus \{1\}$, and let 
\[
A_{\zeta,\psi}^{\vee} \coloneqq
\{\xi \in A_{\zeta}^{\vee} \mid \xi|_{\mathbb{F}_p}
=\psi\}. 
\]
In the following, we take a subfield
$\mathbb{F}_q \subset \mathbb{F}_{p^{2f}}$ and 
assume $R(x) \in \mathbb{F}_q[x]$. 
The $L$-polynomial of $\overline{C}_{R,r}$ over $\mathbb{F}_q$ is defined by
\[
L_{\overline{C}_{R,r}/\mathbb{F}_q}(T)
 \coloneqq
\det (1-\mathrm{Fr}_q^\ast T; H^1(\overline{C}_{R,r})). 
\]

\begin{proposition}
Assume $\mathbb{F}_{p^r} \subset \mathbb{F}_q$ and $A_{\zeta} \subset \mathbb{F}_q^2$
for every $\zeta \in \mathbb{F}_{p^r}^{\times}$.
There exist elements $\{\tau_{\xi} \in \overline{\mathbb{Q}}_{\ell}^{\times}\}_{\xi \in A_{\zeta,\psi}^{\vee}}$ which are $q^{1/2}$ 
times roots of unity such that  
\[
L_{\overline{C}_{R,r}/\mathbb{F}_q}(T)
=\prod_{\zeta \in \mathbb{F}_{p^r}^{\times}}
\prod_{\xi \in A_{\zeta, \psi}^{\vee}} 
(1-\tau_{\xi} T).
\]
Consequently, the curve $\overline{C}_{R,r}$ is supersingular. 
\end{proposition}
\begin{proof}
This is a consequence of 
\eqref{ai}, \cite[Theorem 1.1 and Lemma 2.8]{TT},
and Lemma \ref{maximal minimal}. 
\end{proof}

\begin{lemma}\label{fl}
%\fbox{Assume $\mathbb{F}_{p^r} \subset \mathbb{F}_q$.}
The curve $\overline{C}_{R,r}$
is $\mathbb{F}_q$-maximal if and only if $\mathbb{F}_{p^r} \subset \mathbb{F}_q$ and 
 the curve $\overline{C}_{\zeta R}$ 
 is $\mathbb{F}_q$-maximal 
 for every $\zeta \in \mathbb{F}_{p^r}^{\times}$. 
\end{lemma}

\begin{proof}
Assume that $\overline{C}_{R,r}$
is $\mathbb{F}_q$-maximal. 
By Lemma \ref{ail}, we have $\mathbb{F}_{p^r} \subset \mathbb{F}_q$. 
For any 
$\zeta \in \mathbb{F}_{p^r}^{\times}$, 
we have the finite \'etale morphism 
\begin{equation}\label{cRr}
C_{R,r} \to C_{\zeta R};\ 
(x,z) \mapsto \left(x,\ \sum_{i=0}^{r-1} 
(\zeta z)^{p^i}\right)
\end{equation}
over $\mathbb{F}_q$.
This extends to a finite morphism 
$\overline{C}_{R,r} \to \overline{C}_{\zeta R}$.
This implies that $\overline{C}_{\zeta R}$
is $\mathbb{F}_q$-maximal. 

To the contrary, assume that $\overline{C}_{\zeta R}$
is $\mathbb{F}_q$-maximal for 
every $\zeta \in \mathbb{F}_{p^r}^{\times}$. 
From \eqref{ai}, we have 
\[
H^1(\overline{C}_{R,r})\simeq 
H^1_{\rm c}(C_{R,r})
\simeq \bigoplus_{\zeta \in 
\mathbb{F}_{p^r}^{\times}}
H_{\rm c}^1(\mathbb{A}^1,\mathscr{L}_{\psi}(x \zeta R(x))) \subset 
\bigoplus_{\zeta \in \mathbb{F}_{p^r}^{\times}}
H_{\rm c}^1(C_{\zeta R}) \simeq 
\bigoplus_{\zeta \in \mathbb{F}_{p^r}^{\times}}
H^1(\overline{C}_{\zeta R}).
\]
Thus, the curve $\overline{C}_{R,r}$ is $\mathbb{F}_q$-maximal. 
\end{proof}

\subsection{A criterion for the $\mathbb{F}_{p^{2f}}$-maximality of $\overline{C}_{R,r}$ when $e = f$}

In this subsection, we consider the case $e=f$,
i.e.,\
\[
R(x) = \sum_{i=0}^f a_i x^{p^i} \in \mathbb{F}_{p^{2f}}[x]
  \qquad (a_f \neq 0)
\]
is an $\mathbb{F}_p$-linearized polynomial of degree $p^f$ over $\mathbb{F}_{p^{2f}}$.

The following result is a generalization of Proposition \ref{pp}.

\begin{proposition}\label{c1}
%\fbox{Assume $\mathbb{F}_{p^r} \subset \mathbb{F}_{p^{2f}}$.}
The following conditions are equivalent. 
\begin{itemize}
\item[{\rm (1)}] The curve 
$\overline{C}_{R,r}$ is 
$\mathbb{F}_{p^{2f}}$-maximal. 
\item[{\rm (2)}]
$f \equiv 0 \pmod{r}$,
$a_0=\cdots=a_{f-1}=0$, and 
$a_f^{p^f}+a_f=0$. 
\end{itemize}
\end{proposition}

\begin{proof}
We show $(1) \Rightarrow (2)$. 
%By Lemma \ref{ail}, we have $\mathbb{F}_{p^r} \subset \mathbb{F}_{p^{2f}}$. 
Let $\zeta \in \mathbb{F}_{p^r}^{\times}$ be any element.
By Lemma \ref{fl}, the curve $\overline{C}_{\zeta R}$ is 
$\mathbb{F}_{p^{2f}}$-maximal.
From Proposition \ref{pp}, it follows that  $a_0=\cdots=a_{f-1}=0$ and 
$(\zeta a_f)^{p^f}+\zeta a_f=0$.
By taking $\zeta = 1$, we obtain $a_f^{p^f}+a_f=0$. 
Hence
\[
(\zeta a_f)^{p^f}+\zeta a_f = (-\zeta^{p^f}+\zeta) \cdot a_f = 0
\]
for every $\zeta \in \mathbb{F}_{p^r}$.
Since $a_f \neq 0$, we have $\zeta \in \mathbb{F}_{p^f}$.
Hence $\mathbb{F}_{p^r} \subset \mathbb{F}_{p^f}$.
Therefore, we have $f \equiv 0 \pmod{r}$.

We show $(2) \Rightarrow (1)$. From $f \equiv 0 \pmod{r}$, it follows that $\mathbb{F}_{p^r} \subset\mathbb{F}_{p^{2f}}$. 
We have $(\zeta a_f)^{p^f}+\zeta a_f=0$
for every $\zeta \in\mathbb{F}_{p^r}^{\times}$. 
By Proposition \ref{pp}, the curve 
$\overline{C}_{\zeta R}$ is $\mathbb{F}_{p^{2f}}$-maximal 
for every $\zeta \in \mathbb{F}_{p^r}^{\times}$. 
Thus (1) follows from Lemma \ref{fl}. 
\end{proof}

\subsection{A conjecture on the maximality of the curve $\overline{C}_{R,r}$}

Assume $e \geq 1$.
Let
\[
R(x) = \sum_{i=0}^e a_i x^{p^i} \in \mathbb{F}_{p^{2f}}[x]
  \qquad (a_e \neq 0)
\]
be an $\mathbb{F}_p$-linearized polynomial of degree $p^e$ over $\mathbb{F}_{p^{2f}}$.
We put 
\[
e_R \coloneqq \mathrm{gcd}\{i \geq 1 \mid a_i \neq 0\}. 
\]

We propose the following conjecture.

\begin{conjecture}\label{conj}
If $e \geq 1$ and 
$\overline{C}_{R,r}$ is $\mathbb{F}_{p^{2f}}$-maximal,
then $\mathrm{gcd}(e_R, f) \equiv 0 \pmod{r}$. 
\end{conjecture}

We give some evidences for 
Conjecture \ref{conj}.

\begin{proposition}\label{mon}
Conjecture \ref{conj} holds true if one of the following conditions holds. 

\begin{itemize}
\item[{\rm (1)}] 
$e=f$. 
\item[{\rm (2)}] 
$R(x)$ is a monic polynomial of degree $p^e$, i.e.,\ $R(x) = a_e x^{p^e}$ for some $a_e \in \mathbb{F}_{p^{2f}}^{\times}$. 
\end{itemize}
\end{proposition}
\begin{proof}
(1) It follows from 
Proposition \ref{c1}. 

(2) Since $\overline{C}_{R,r}$ is $\mathbb{F}_{p^{2f}}$-maximal,
we have
$e \equiv 0 \pmod{r}$ and
$f \equiv 0 \pmod{e}$
by Coulter's results \cite[Theorem 8.12]{C}. 
Hence the claim (2) follows. 
\end{proof}

\begin{comment}
\begin{question}
If $p \equiv 3 \pmod 4$ and 
$R(x)=2x^p-x$, 
the curve $\overline{C}_R$ is 
$\mathbb{F}_{p^{6p_0^2}}$-maximal
by Theorem \ref{2R}. 
We expect that $\overline{C}_{R,2}$ 
and $\overline{C}_{R,3}$ are \textit{not}  
$\mathbb{F}_{p^{6p_0^2}}$-maximal?
\end{question}
\end{comment}

\section{Further examples of maximal curves}
\label{Further examples}

In this section, we assume $p_0 \neq 2$, and consider the following $\mathbb{F}_p$-linearized polynomial
\[ R(x) \coloneqq 2x^p+x. \]

Our main aim 
in this section is to study the maximality of
the curve $\overline{C}_{R}$ and its twists
$\{\overline{C}_{\zeta R}\}_{\zeta \in \mathbb{F}^{\times}}$
in some detail.
We also study the curves $\overline{C}_{R,r}$ for $r \geq 1$.
We find maximal curves among them.
Using Lemma \ref{crl},
we also find integers $r$ such that  
the curve $\overline{C}_{R,r}$ is not 
$\mathbb{F}_{p^k}$-maximal for any $k \geq 1$.

%In this process, we give a criterion for a curve to be not $\mathbb{F}_{p^n}$-maximal for any $n \geq 1$ in Lemma \ref{crl}.

\subsection{Criteria for the maximality}
\label{subsection criterion}

Assume $p_0 \neq 2$, and let $R(x) \coloneqq 2x^p+x$.
Let $\zeta \in \mathbb{F}^{\times}$.
We have
\[
  E_{\zeta R}(x) = 2 ( \zeta^p x^{p^2} + \zeta^p x^p + \zeta x),
\]
and
\begin{equation}\label{ed}
  V_{\zeta R} = \{ x \in \mathbb{F} \mid \zeta^p x^{p^2} + \zeta^p x^p + \zeta x = 0\}.
\end{equation}
Take an element $\xi \in V_{\zeta R} \setminus \{0\}$.
Then we have
\[
  \zeta^p \xi^{p^2} + \zeta^p \xi^p + \zeta \xi = 0.
\]
We take a positive integer $d \geq 1$
such that $\xi, \zeta \in \mathbb{F}_{p^d}$.

\begin{lemma}\label{quot}
The curve $C_{\zeta R}$ is isomorphic to the curve $C'$
over $\mathbb{F}_{p^d}$ defined by the equation
\[
  y^p-y = -\zeta \xi^{p+1}(x^p-x)^2.
\]
\end{lemma}
\begin{proof}
Let $\eta=-\zeta/\xi^{p-1}$. 
We have the isomorphism 
\[
C_{\zeta R} \xrightarrow{\sim} C';\ 
(x,z) \mapsto (\xi^{-1}x,z+\eta^{1/p}x^2).
\]
It suffices to check that this is well-defined.  
By \eqref{ed}, 
we have $(\zeta\xi)^p+\zeta \xi=-\zeta^p\xi^{p^2}$. 
Dividing this by 
$\xi^p$, substituting 
$\eta=-\zeta/\xi^{p-1}$ and 
taking the $p$-th root, 
we obtain $\zeta-\eta^{1/p}=\zeta^2/\eta$.  
By setting $y=z+\eta^{1/p} x^2$, 
we compute 
\begin{align*}
y^p-y &= z^p-z+\eta x^{2p}-\eta^{1/p}x^2 \\
&=2 \zeta x^{p+1}+(\zeta-\eta^{1/p}) x^2+\eta x^{2p} \\
&=2 \zeta x^{p+1}+(\zeta^2/\eta) x^2+\eta x^{2p} \\
&=\eta(x^p+(\zeta/\eta)x)^2.
\end{align*}
Further, by putting $x_1=\xi^{-1}x$, we have 
$
(x^p+(\zeta/\eta)x)^2=\xi^{2p}(x_1^p-x_1)^2. 
$
Hence the claim follows from 
$\eta \xi^{2p}=-\zeta \xi^{p+1}$. 
\end{proof}
For a representation $M$ over 
$\overline{\mathbb{Q}}_{\ell}$
of a finite abelian group 
$G$ and a character $\chi \in G^{\vee}$, let 
$M[\chi]$ denote the $\chi$-isotypic part of 
$M$. 
\begin{lemma}\label{imp}
 
\begin{itemize}
\item[{\rm (1)}]
We have isomorphisms 
\[
H^1(\overline{C}_{\zeta R})
\xleftarrow{\sim}
H_{\rm c}^1(C_{\zeta R})
\simeq 
\bigoplus_{\psi \in 
\mathbb{F}_p^{\vee} \setminus \{1\}}
H_{\rm c}^1(\mathbb{A}^1,\mathscr{L}_{\psi}(x \zeta R(x))). 
\]
\item[{\rm (2)}] 
Let $\psi \in \mathbb{F}_p^{\vee}\setminus \{1\}$.
We have an isomorphism
\[
H_{\rm c}^1(\mathbb{A}^1,\mathscr{L}_{\psi}(x \zeta R(x))) \simeq 
\bigoplus_{a \in \mathbb{F}_p}
H_{\rm c}^1(\mathbb{A}^1,\mathscr{L}_{\psi}
(-\zeta \xi^{p+1} u^2+au)). 
\]
\item[{\rm (3)}] 
Let $\psi \in \mathbb{F}_p^{\vee}\setminus \{1\}$ and $a \in \mathbb{F}_p$.
The Frobenius element
$\mathrm{Fr}_{p^d}^\ast$ acts on
\[ H_{\rm c}^1(\mathbb{A}^1,\mathscr{L}_{\psi}
(-\zeta \xi^{p+1} u^2+au)) \]
as the scalar multiplication by 
\[
\psi_{p^d}(4^{-1} \zeta^{-1} \xi^{-(p+1)} a^2) \left(\frac{-\zeta}{p^d}\right) G(\psi_{p^d}). 
\]
\end{itemize}
\end{lemma}

\begin{proof}
(1) It follows from \cite[Lemma 2.8 and 
Corollary 3.6]{TT}. 

(2) By Lemma \ref{quot}, 
we have an isomorphism 
$H_{\rm c}^1(C_{\zeta R})[\psi] \simeq 
H_{\rm c}^1(C')[\psi]$. 
Let $\pi \colon 
\mathbb{A}^1 \to \mathbb{A}^1;\ x_1 \mapsto
u \coloneqq x_1^p-x_1$. 
We note 
\[
\pi_{\ast}\overline{\mathbb{Q}}_{\ell}
\simeq \bigoplus_{\psi' \in \mathbb{F}_p^{\vee}}
\mathscr{L}_{\psi'}(u).
\]
From the projection formula, it follows that 
\begin{align*}
H_{\rm c}^1(\mathbb{A}^1,\mathscr{L}_{\psi}(x \zeta R(x))) &\simeq H_{\rm c}^1(C_{\zeta R})[\psi] \\ 
 &\simeq H_{\rm c}^1(C')[\psi]  \\
 & \simeq 
H_{\rm c}^1(\mathbb{A}^1,\pi^\ast
\mathscr{L}_{\psi}(-\zeta \xi^{p+1} u^2)) \\
& \simeq \bigoplus_{\psi' \in \mathbb{F}_p^{\vee}}
H_{\rm c}^1(\mathbb{A}^1,
\mathscr{L}_{\psi}(-\zeta \xi^{p+1} u^2) \otimes 
\mathscr{L}_{\psi'}(u)) \\
&\simeq 
\bigoplus_{a \in \mathbb{F}_p}
H_{\rm c}^1(\mathbb{A}^1,\mathscr{L}_{\psi}
(-\zeta \xi^{p+1} u^2+au)). 
\end{align*}

(3) Let
\[
  V_a \coloneqq H_{\rm c}^1(\mathbb{A}^1,\mathscr{L}_{\psi}
(-\zeta \xi^{p+1} u^2+au)).
\]
It is well-known that $ \dim V_a=1$ (cf.\ \cite[Lemma 3.3]{TT} for $e=0$). 
By the Grothendieck trace formula, $\mathrm{Fr}^\ast_{p^d}$
acts on $V_a$ as the scalar multiplication by 
\[
-\sum_{u \in \mathbb{F}_{p^d}}
\psi_{p^d}(-\zeta\xi^{p+1} u^2+a u). 
\]
This value equals 
\[
\psi_{p^d}(4^{-1} \zeta^{-1} \xi^{-(p+1)} a^2)  
\left(\frac{-\zeta \xi^{p+1}}{p^d}\right) G(\psi_{p^d}). 
\]
Here, we have $\displaystyle \left( \frac{\xi^{p+1}}{p^d} \right) = 1$ since $p+1$ is even.
The claim follows. 
\end{proof}

\begin{corollary}\label{ccc}
\begin{itemize}
\item[{\rm (1)}]
If $p^d \equiv 3 \pmod 4$,
the curve $\overline{C}_{\zeta R}$ is 
$\mathbb{F}_{p^{2p_0 d}}$-maximal. 

\item[{\rm (2)}] 
If $d$ is even and 
\[
\left(\frac{\zeta}{p^d}\right)
\left(\frac{-1}{p}\right)^{d/2}
= -1,
\]
the curve $\overline{C}_{\zeta R}$ 
is $\mathbb{F}_{p^{p_0 d}}$-maximal. 
\end{itemize}
\end{corollary}

\begin{proof}
(1) Since $p^d \equiv 3 \pmod 4$, we have
$\displaystyle
G(\psi_{p^{d}})^2=\displaystyle \left(\frac{-1}{p^d}\right) p^d=-p^d$.
The claim follows from Lemma \ref{imp} (3)
and $\psi_{p^d}^{p_0}=1$. 

(2) Since $d$ is even, we have
$\displaystyle
G(\psi_{p^d})
=\left(\frac{-1}{p}\right)^{d/2}
p^{d/2}$
and
$\displaystyle \left(\frac{-\zeta}{p^d}\right)=
\left(\frac{\zeta}{p^d}\right)$
by Lemma \ref{HG}.
The claim follows from Lemma \ref{imp} (3). 
\end{proof}

\subsection{Maximality of the twists $\overline{C}_{\zeta R}$}

In this subsection,
we study elements $\xi, \zeta$ satisfying
the conditions in Subsection \ref{subsection criterion}.
We find some maximal curves among
the twists $\overline{C}_{\zeta R}$.

Let $\alpha \in \mathbb{F}_p^{\times} \setminus 
\{-1\}$. 
We set
\[ \beta \coloneqq -(\alpha(\alpha+1))^{-1}. \]
Let $\xi, \zeta \in \mathbb{F}^{\times}$ be elements satisfying
\begin{equation}\label{xz}
\xi^{p-1} = \alpha, \qquad \zeta^{p-1} = \beta. 
\end{equation}

\begin{lemma}\label{lld}
\begin{itemize}
\item[{\rm (1)}] 
We have $\zeta^p \xi^{p^2} + \zeta^p \xi^p + \zeta \xi = 0$. 
\item[{\rm (2)}] Let $d_1$ be the order of $\alpha$
in $\mathbb{F}_p^{\times}$ and, 
let $d_2$ be the order of 
$\beta$ in $\mathbb{F}_p^{\times}$.   
We have $\mathbb{F}_p(\xi) = \mathbb{F}_{p^{d_1}}$
and $\mathbb{F}_{p}(\zeta) = \mathbb{F}_{p^{d_2}}$. 
\end{itemize}
\end{lemma}

\begin{proof}
(1) Since $\alpha \in \mathbb{F}_p$, we have $\alpha^p = \alpha$.
From \eqref{xz}, we calculate
\begin{align*}
(\zeta\xi^p)^p+(\zeta \xi)^p+\zeta \xi
&= \zeta^p(\xi^p+\xi)^p+\zeta \xi \\
&= \zeta^p((\alpha+1)\xi)^p+\zeta \xi \\
&= \zeta^p \alpha (\alpha+1)^p \xi+\zeta \xi \\
& =\xi\alpha(\alpha+1)(\zeta^p-\beta \zeta) \\
&=0. 
\end{align*}

(2) For every integer $i \geq 1$, we have 
$\xi^{p^i}=\alpha^i \xi$ by $\alpha \in 
\mathbb{F}_p$. 
By the definition of $d_1$,
we have $\xi \in \mathbb{F}_{p^{i}}$
(i.e.,\ we have $\xi^{p^i} = \xi$)
if and only if $i \equiv 0 \pmod{d_1}$.
Hence we have $\mathbb{F}_p(\xi) = \mathbb{F}_{p^{d_1}}$.
The assertion for $\zeta$ is proved similarly.
\end{proof}

Let $d \coloneqq \mathrm{lcm}(d_1,d_2)$.
By Lemma \ref{lld}, we have $\mathbb{F}_{p^d}=\mathbb{F}_p(\xi,\zeta)$. 

\begin{lemma}\label{35}
If $d$ is even, we have 
\[
\left(\frac{\zeta}{p^d}\right)=
\left(\beta^{d/2}\right)^{(p+1)/2}. 
\]
\end{lemma}

\begin{proof}
Since $d$ is even, we have 
$\beta^{d/2} \in \{\pm 1\}$ by $\beta^d = (\beta ^{d_2})^{d/d_2} = 1$.   
We have
$\zeta^{d(p-1)/2}
=(\zeta^{p-1})^{d/2}=\beta^{d/2}$. 
Therefore, by $\zeta^{p^i} = \zeta \cdot \beta^i$
for an integer $i \geq 1$, we compute 
\begin{align*}
\left(\frac{\zeta}{p^d}\right) &=
\left(\zeta^{p^{d-1}+\cdots+p+1}\right)^{(p-1)/2} 
= \left(\zeta^{d} \cdot \beta^{(d-1)+\cdots+2+1}
\right)^{(p-1)/2} \\
&= \zeta^{d(p-1)/2} \cdot ( \beta^{(d-1)d/2}
)^{(p-1)/2}
= \beta^{d/2} \cdot ( ( \beta^{d/2} )^{d-1} 
)^{(p-1)/2}
\end{align*}
Since $d-1$ is odd and $\beta^{d/2} \in \{\pm 1\}$,
we have $( \beta^{d/2} )^{d-1} = \beta^{d/2}$.
Therefore, we have
\[
\left(\frac{\zeta}{p^d}\right) 
= \beta^{d/2} \cdot ( \beta^{d/2}
)^{(p-1)/2}
= ( \beta^{d/2}
)^{(p+1)/2}.
\]
\end{proof}

\begin{theorem}\label{mp}
Assume $p_0 \neq 2$, and put $R(x)=2x^p+x$.
Let $\alpha \in \mathbb{F}_p^{\times} \setminus 
\{-1\}$. 
We set $\beta \coloneqq -(\alpha(\alpha+1))^{-1}$.
Let $\xi, \zeta \in \mathbb{F}^{\times}$ be elements satisfying
$\xi^{p-1} = \alpha$ and $\zeta^{p-1} = \beta$.
Let $d_1$ be the order of $\alpha$
in $\mathbb{F}_p^{\times}$, and
$d_2$ the order of 
$\beta$ in $\mathbb{F}_p^{\times}$.
We put $d \coloneqq \mathrm{lcm}(d_1,d_2)$.
\begin{itemize}
\item[{\rm (1)}]
If $d$ is odd and $p \equiv 3 \pmod 4$,
the curve $\overline{C}_{\zeta R}$ is 
$\mathbb{F}_{p^{2p_0d}}$-maximal.

\item[{\rm (2)}]
If $d$ is even and $\beta^{d/2}=-1$,
the curve $\overline{C}_{\zeta R}$ is 
$\mathbb{F}_{p^{p_0d}}$-maximal. 
\end{itemize}
\end{theorem}

\begin{proof}
(1) By Lemma \ref{lld},
the elements $\xi, \zeta \in \mathbb{F}^{\times}$
satisfy the conditions in Subsection \ref{subsection criterion}.
Therefore, the claim follows from Corollary \ref{ccc} (1).

(2) By Lemma \ref{35}, we have
 \begin{equation}\label{beta}
 \left(\frac{\zeta}{p^d}\right)=(-1)^{(p+1)/2} = -\left(\frac{-1}{p}\right)
 \end{equation}
since $\beta^{d/2}=-1$.

If $p \equiv 1 \pmod 4$, we have 
$\displaystyle \left(\frac{-1}{p}\right)=1$,
and the claim follows from Corollary \ref{ccc} (2).

If $p \equiv 3 \pmod4$, the integer $d/2$ is odd
since $d$ divides $|\mathbb{F}_p^{\times}| = p-1$.
Hence we have
\[
  \left(\frac{\zeta}{p^d}\right)
  \left(\frac{-1}{p}\right)^{d/2}
  =
  - \left(\frac{-1}{p}\right)
  \left(\frac{-1}{p}\right)
  = -1,
\]
and the claim follows from Corollary \ref{ccc} (2).
\end{proof}

\begin{corollary}\label{-2}
Let  
$\alpha=1$, $\beta = - 2^{-1}$, and $\xi=1$.
Let $\zeta \in \mathbb{F}^{\times}$
be an element satisfying $\zeta^{p-1} = - 2^{-1}$.
Let $d$ be the order of $-2$ in $\mathbb{F}_{p_0}^{\times}$.
\begin{itemize}
\item[{\rm (1)}] 
If $d$ is odd and $p \equiv 3 \pmod 4$, 
the curve $\overline{C}_{\zeta R}$ 
is $\mathbb{F}_{p^{2p_0d}}$-maximal. 
If $p_0 \equiv 3 \pmod{8}$, 
the order $d$ is odd.  

\item[{\rm (2)}] 
If $d$ is even, 
the curve $\overline{C}_{\zeta R}$ 
is $\mathbb{F}_{p^{p_0d}}$-maximal. 
If $p_0 \equiv 5,7 \pmod{8}$, the order 
$d$ is even. 
\end{itemize}
\end{corollary}

\begin{proof}
Note that $\beta^{d/2} = -1$ because $d$ is equal
to the order of $\beta$ in $\mathbb{F}_p^{\times}$.
The former claims in (2) and (3) are direct consequences of 
 Theorem \ref{mp}. 
 We recall that 
 \[
 \left(\frac{-2}{p_0}\right)=(-1)^{(p_0-1)(p_0+5)/8} = 1 \iff p_0 \equiv 1,3 \pmod{8}. 
 \]
 
We show the latter claim in (2). 
The assumption $p_0 \equiv 3 \pmod{8}$ implies
$\displaystyle \left( \frac{-2}{p_0} \right) = 1$.
Hence the order $d$ divides $(p_0-1)/2$, and 
$(p_0-1)/2$ is odd. Thus $d$ is odd. Hence the claim follows. 

We show the latter claim
in (3). 
The assumption 
$p_0 \equiv 5,7 \pmod{8}$ implies $\displaystyle \left(\frac{-2}{p_0}\right)=-1$. 
Hence $d$ divides $p_0-1$, and $(p_0 - 1)/d$ is odd.
Since $p_0 - 1$ is even, the order $d$ is even.
\end{proof}

Using Corollary \ref{-2},
it is easy to give explicit maximal curves.
Here we give an example.

\begin{example}
Let $p_0=11$ and $\alpha=3$.
Then $\beta=-1$, $d_1=5$, and $d_2=2$.  
Hence
$d = \mathrm{lcm}(d_1,d_2) = 10$.
Let $\zeta \in \mathbb{F}$
be an element satisfying $\zeta^{10} = -2^{-1}$.
Since $d$ is even and $\beta^{d/2}=-1$, 
the curve $\overline{C}_{\zeta R}$ is 
$\mathbb{F}_{11^{110}}$-maximal by Theorem
\ref{mp} (2). 
\end{example}

\subsection{Maximality of the curve $\overline{C}_{R}$}

In the next theorem, we study the maximality of
the van der Geer--van der Vlugt curve $\overline{C}_{R}$.

We put $\zeta=1$. 
Then $\xi$ satisfies $\xi^{p^2}+\xi^p+\xi
=0$ by \eqref{ed}.
Hence we have
\[
  \xi^{p^3}-\xi = 
  (\xi^{p^2}+\xi^p+\xi)^p-(\xi^{p^2}+\xi^p+\xi)=0.
%  (\xi^{p^2})^p  = (- \xi^p - \xi)^p  = - \xi^{p^2} - \xi^p  = \xi.
\]
Thus $\xi \in \mathbb{F}_{p^3}$.

\begin{theorem}\label{lc}
\begin{itemize}
\item[{\rm (1)}]
Assume $p \equiv 1 \pmod 4$. 
The curve $\overline{C}_R$ is $\mathbb{F}_{p^k}$-minimal
if $k \equiv 0 \pmod{6}$.

\item[{\rm (2)}] 
Assume $p \equiv 3 \pmod 4$.
The curve $\overline{C}_R$ is $\mathbb{F}_{p^k}$-maximal
if and only if $k \equiv 0 \pmod{6}$ and $k/6$ is odd.
\end{itemize}
\end{theorem}

\begin{proof}
For any $a \in \mathbb{F}_p$, 
we have 
\[
\Tr_{{p^3}/p}\left(4^{-1} \xi^{-(p+1)} a^2\right)=
\frac{a^2}{4}\Tr_{{p^3}/p}
(\xi^{-(p+1)})=\frac{a^2}{4} \cdot \frac{\xi^{p^2}+\xi^p+\xi}{\xi^{p^2+p+1}}=0
\]
by $\xi^{p^2}+\xi^p+\xi=0$. 

Let $\psi \in \mathbb{F}_p^{\vee}\setminus \{1\}$ and $a \in \mathbb{F}_p$.
By Lemma \ref{imp} (3) for $\zeta=1$ and $d = 6$,
the Frobenius element
$\mathrm{Fr}_{p^6}^\ast$ acts on
$H_{\rm c}^1(\mathbb{A}^1,\mathscr{L}_{\psi}
(-\xi^{p+1} u^2+au))$
as the scalar multiplication by 
\begin{align*}
& \psi_{p^6}(4^{-1} \xi^{-(p+1)} a^2) \left(\frac{-1}{p^6}\right) G(\psi_{p^6}) \\
&= 
\psi \big( 2 \Tr_{{p^3}/p}(4^{-1} \xi^{-(p+1)} a^2) \big) \cdot (-1)^{(p^6-1)/2} \cdot \left(\frac{-1}{p}\right)^3 p^3 \\
&= 
\left(\frac{-1}{p}\right) p^3.
\end{align*}
For a positive integer $k$ with $k \equiv{0} \pmod{6}$,
the Frobenius element
$\mathrm{Fr}_{p^k}^\ast$
acts on $H^1_{\rm c}(C_R) \simeq H^1(\overline{C}_R)$
as the scalar multiplication by
$\displaystyle \left(\frac{-1}{p}\right)^{k/6} p^{k/2}$.
Hence the claims follow. 

To the contrary, we assume that $\overline{C}_R$ is 
$\mathbb{F}_{p^k}$-maximal. By the above argument, 
every eigenvalue of $\mathrm{Fr}_p^\ast$ on $H^1(\overline{C}_R)$
has the form $\zeta_1 \zeta_2 p^{1/2}$ with a primitive 
$4$-th root of unity $\zeta_1$ and a third root of unity $\zeta_2$. 
Hence $k$ is even and $4$ does not divide $k$. By Proposition \ref{split}, we have 
$V_R \subset \mathbb{F}_{p^k}$ and 
$\mathbb{F}_{p^3} \subset \mathbb{F}_{p^k}$ by \eqref{ed}.
Thus $3$ divides $k$. Hence $k$ is a multiple of $6$ and $k/6$ is 
odd. 
\end{proof}

\begin{remark}
When $p_0=3$ or $p_0 \equiv 7 \pmod{12}$, the `if' part in 
Theorem \ref{lc} is shown by \"Ozbudak--Sayg\i \  by a different method; see \cite[Theorem 3.10]{OS}.
\end{remark}

\begin{remark}
Let $R_-(x) \coloneqq -2x^p+x$. 
We have an $\mathbb{F}_{p^2}$-isomorphism 
$C_{R_-} \simeq C_R$ by Example \ref{2R} (1) for $e=1$. 
Hence the same claim as Theorem \ref{lc} 
 holds for $\overline{C}_{R_-}$. 
\end{remark}

\subsection{Non-maximality of the curve $\overline{C}_{R,r}$ for some $r \geq 1$}

In this subsection, we show that for some $r \geq 1$,
the curve $\overline{C}_{R,r}$ is 
not maximal over any finite extension of $\mathbb{F}_{p_0}$.

\begin{theorem}\label{lcc}
Let the notation be as in Theorem \ref{mp}. 
Assume that there exist $\alpha,\beta\in \mathbb{F}_p^{\times}$ and $\xi,\zeta \in \mathbb{F}_{p^d}^{\times}$ 
such that $d \equiv 0 \pmod 4$ and 
$\beta^{d/2}=-1$. Let $r \geq 1$ be a multiple of $d$.  
The curve $\overline{C}_{R,r}$ is 
not $\mathbb{F}_{p_0^k}$-maximal for any integer $k \geq 1$. 
\end{theorem}

\begin{proof}
By Corollary \ref{aic}, 
it suffices to show that $\overline{C}_{R,r}$ is 
not $\mathbb{F}_{p^k}$-maximal for any integer $k \geq 1$. 

Take $\alpha,\beta\in \mathbb{F}_p^{\times}$ and $\xi,\zeta \in \mathbb{F}_{p^d}^{\times}$
satisfying the assumption.
Recall that $d_1$ and $d_2$ are the orders of 
$\alpha \in \mathbb{F}_p^{\times}$ and 
$\beta \in \mathbb{F}_p^{\times}$, respectively. 
Hence $p-1$ is divisible by both $d_1$ and $d_2$. 
Since $d=\mathrm{lcm}(d_1,d_2)$, the integer 
$d$ divides $p-1$.
Thus  
the condition $d \equiv 0 \pmod{4}$ 
implies $p \equiv 1 \pmod 4$. 

We will apply Lemma \ref{crl}. 
By \eqref{cRr} for $\zeta$ in the assumption, we have a 
finite morphism $\overline{C}_{R,r} \to \overline{C}_{\zeta R}$
defined over $\mathbb{F}_{p^d}$. 
By Theorem \ref{mp} (2),  the curve $\overline{C}_{\zeta R}$ is $\mathbb{F}_{p^{p_0d}}$-maximal.   We note that 
$p_0 d \equiv 0 \pmod 4$ since $d \equiv 0 \pmod 4$. 

By \eqref{cRr} for $\zeta=1$, we have a 
finite morphism $\overline{C}_{R,r} \to \overline{C}_R$
defined over $\mathbb{F}_p$. 
From $p \equiv 1 \pmod 4$ and Theorem \ref{lc} (1), it follows that 
$\overline{C}_R$ is 
$\mathbb{F}_{p^6}$-minimal. 

By applying Lemma \ref{crl} with 
$C=\overline{C}_{R,r}$,  
$C_1=\overline{C}_R$, $n_1=6$, 
$C_2=\overline{C}_{\zeta R}$ and $n_2=p_0 d$, we obtain the claim. 
%\fbox{??}
\end{proof}

We give examples which fall into the situation 
in Theorem \ref{lcc}. 

\begin{corollary}
\label{lcc2}
\begin{itemize}
\item[{\rm (1)}] 
Assume $p_0\equiv 5\pmod{8}$.
Let $d$ be the order of $-2$ in $\mathbb{F}_{p_0}^{\times}$. Let $r \geq 1$
be a multiple of $d$. 
The curve $\overline{C}_{R,r}$ is 
not $\mathbb{F}_{p_0^k}$-maximal for any integer $k \geq 1$. 
\item[{\rm (2)}] 
Let $p_0=2^{2^m}+1\ (m \geq 1)$
be a Fermat prime. Let $r \geq 1$ be a multiple of 
$2^{m+1}$. 
The curve $\overline{C}_{R,r}$ is 
not $\mathbb{F}_{p_0^k}$-maximal for any integer $k \geq 1$. 
\end{itemize}
\end{corollary}

\begin{proof}
We consider the situation 
in  Corollary \ref{-2}.
Namely, we have $\alpha=1$, $\beta = - 2^{-1}$, $\xi=1$,
$\zeta \in \mathbb{F}^{\times}$ with $\zeta^{p-1} = - 2^{-1}$,
and $d$ is the order of $-2$ in $\mathbb{F}_{p_0}^{\times}$.

(1) The assumption $p_0 \equiv 5 \pmod{8}$ implies  
$\displaystyle \left(\frac{-2}{p_0}\right)=-1$. 
If $d$ is odd, the integer 
$d$ divides $(p_0-1)/2$ 
by $p_0-1 \equiv 0 \pmod{d}$. 
By $(-2)^d=1 \in \mathbb{F}_{p_0}^{\times}$, 
we have $(-2)^{(p_0-1)/2}=
((-2)^d)^{(p_0-1)/(2d)}=1$, 
which contradicts $\displaystyle \left(\frac{-2}{p_0}\right)=-1$. 
Thus $d$ is even. 
We write $d = 2d'$. 
Since $d$ is the order of $-2$, 
we have $(-2)^{d'}=-1$. 
Since $d'$ divides $(p_0-1)/2$ and 
$(-2)^{(p_0-1)/2}=-1$, the integer  
$(p_0-1)/(2d')$ is odd. Thus 
$d'$ is even by $p_0 \equiv 1 \pmod 4$. 
Hence $4$ divides $2d' = d$. 
Since $\alpha=1$, $\beta = - 2^{-1}$, $\xi=1$, and $\zeta$
satisfy the assumption of Theorem \ref{lcc},
the claim follows. 

(2) From $p_0=2^{2^m}+1$, it follows that 
$d=2^{m+1}$. 
Since $d=2^{m+1} \equiv 0 \pmod{4}$ by $m \geq 1$, 
the claim follows from Theorem \ref{lcc}. 
\end{proof}

\subsection{Some examples of maximal curves in characteristic $3$}

Here we give examples of maximal curves
in characteristic $3$.
It is possible to obtain similar examples in any odd characteristics.

We put $p_0 = p =3$. 
Let $\zeta \in \mathbb{F}$ be an element 
such that $\zeta^4+\zeta^2-1=0$. 
Then $\zeta^8=(1-\zeta^2)^2=-1$. 
Thus $\zeta \in \mathbb{F}_{3^4}^{\times}$. 
Let $\xi=\zeta^2$. 

\begin{lemma}
$\zeta^3 \xi^9 + \zeta^3 \xi^3 + \zeta \xi = 0$.
\end{lemma}

\begin{proof}
We have 
$\zeta^3 \xi^9 + \zeta^3 \xi^3 + \zeta \xi
=\zeta^3(\xi^9+\xi^3+\zeta^{-2}\xi)$ and 
\[
\xi^9+\xi^3+\zeta^{-2} \xi
=\xi^9+\xi^3+1=(\xi^3+\xi+1)^3.
\]
Thus it suffices to check 
$\xi^3+\xi+1=0$. 
We have 
\[
\xi^3+\xi+1=
\zeta^6+\zeta^2+1=
(\zeta^2-1)(\zeta^4+\zeta^2-1)=0.
\]
Hence the claim follows.
\end{proof}

\begin{corollary}\label{3c}
Let $k \geq 1$ be a positive integer. 
The curve $\overline{C}_{\zeta R}$ is 
$\mathbb{F}_{3^{4k}}$-maximal if and only if $k$ is odd. 
\end{corollary}

\begin{proof}
We apply Lemma \ref{imp} (3) with $p=3$ and $d=4$. 
By $\zeta^8=-1$, we have 
$\displaystyle \left(\frac{-\zeta}{3^4}\right)=(-\zeta)^{40} = -1$.
Lemma \ref{HG} implies $G(\psi_{3^4})=3^2$. 
By $\xi^4=\zeta^8=-1$, we calculate
\[
\Tr_{{3^4}/3}\left(4^{-1} \zeta^{-1} \xi^{-4} a^2\right)=
-\frac{a^2}{4}\Tr_{{3^4}/3}\left(\zeta^{-1}\right)
=-\frac{a^2}{4}\cdot \frac{\zeta^{26}+\zeta^{24}+\zeta^{18}+1}{\zeta^{27}}=0
\]
for every $a \in \mathbb{F}_3$. 
Thus $\mathrm{Fr}_{3^4}^\ast$ acts on $H^1(\overline{C}_{\zeta R})$ as scalar 
multiplication by $-3^2$ by Lemma \ref{imp}.
The curve $\overline{C}_{\zeta R}$ is  $\mathbb{F}_{3^{4}}$-maximal by Lemma \ref{maximal minimal} (2).
The claim follows from Lemma \ref{maximal minimal extension} (2). 
\end{proof}

\begin{corollary}
Assume $r \equiv 0 \pmod{4}$.
The curve $\overline{C}_{R,r}$ is 
not $\mathbb{F}_{3^k}$-maximal for any integer $k \geq 1$. 
\end{corollary}
\begin{proof}
We apply Lemma \ref{crl}.
Let $\zeta$ be as above. 
By \eqref{cRr}, we have a finite morphism 
$\overline{C}_{R,r} \to \overline{C}_{\zeta R}$
over $\mathbb{F}_{3^4}$. 
The curve $\overline{C}_{\zeta R}$ is 
$\mathbb{F}_{3^4}$-maximal by Corollary \ref{3c}. 

On the other hand, by \eqref{cRr}, we have a finite morphism 
$\overline{C}_{R,r} \to \overline{C}_{R}$ over $\mathbb{F}_3$.
The curve $\overline{C}_R$ is $\mathbb{F}_{3^6}$-maximal 
by Proposition \ref{lc} (2) for $p=3$. 

By applying Lemma \ref{crl} with 
$C=\overline{C}_{R,r}$,  
$C_1=\overline{C}_R$, $n_1=6$, 
$C_2=\overline{C}_{\zeta R}$ and $n_2=4$, we obtain 
the claim. 
\end{proof}

\subsection*{Acknowledgement} 

The first author is supported by JSPS KAKENHI Grant Numbers 23K20786, 23K20204, 21K18577, and 24K21512.
The third author is supported by JSPS KAKENHI Grant Numbers 20K03529 and 23K20786.


\begin{thebibliography}{SGA5}
\bibitem{B}
Bartoli, D., Quoos, L., Sayg\i, Z., Y\i lmaz, E.,
\textit{Explicit maximal and minimal curves of Artin-Schreier type from quadratic forms},
Appl.\ Algebra Engrg.\ Comm. Comput.\ \textbf{32} (2021), no.\ 4, 507-520.

\bibitem{BHMSSV}
Bouw, I., Ho, W., Malmskog, B., Scheidler, R., Srinivasan, P., Vincent, C.,
\textit{Zeta functions of a class of Artin-Schreier curves with many automorphisms},
in Directions in number theory, 87-124, Assoc.\ Women Math.\ Ser., 3, Springer, 2016.

\bibitem{CO}
\c Cak\c cak, E., \"Ozbudak, F.,
\textit{Some Artin-Schreier type function fields over finite fields with prescribed genus and number of rational places},
J.\ Pure Appl.\ Algebra {\bf 210} (2007), no.~1, 113-135.

\bibitem{C}
Coulter, R.\ S.,
\textit{The number of rational points of a class of Artin-Schreier curves},
Finite Fields Appl.\ \textbf{8} (2002), no.\ 4, 397-413.

\bibitem{SommesTrig}
Deligne, P.,
\textit{Applications de la formule des traces aux sommes trigonom\'etriques},
in Cohomologie \'etale, 168-232.\ Lecture Notes in Math., 569,
Springer-Verlag, Berlin, 1977.

\bibitem{GV}
van der Geer, G., van der Vlugt, M.,
\textit{Reed-Muller codes and supersingular curves.\ I},
Compositio Math.\ \textbf{84} (1992), no.\ 3, 333-367.

\bibitem{Hasse-Dapenport}
Hasse, H., Davenport, H.,
\textit{Die Nullstellen der Kongruenzzetafunktionen in gewissen zyklischen F\"allen},
J.\ Reine Angew.\ Math.\ \textbf{172} (1935), 151-182.

\bibitem{Ireland-Rosen}
Ireland, K., Rosen, M.,
\textit{A classical introduction to modern number theory},
Second edition,
Grad.\ Texts in Math., 84, Springer-Verlag, New York, 1990.

\bibitem{Lidl-Niederreiter}
Lidl, R., Niederreiter, H.,
\textit{Finite fields},
Second edition, Encyclopedia Math.\ Appl., \textbf{20},
Cambridge University Press, Cambridge, 1997.

\bibitem{OS}
\"Ozbudak, F., Sayg\i, Z.,
\textit{Explicit maximal and minimal curves over finite fields of odd characteristics},
Finite Fields Appl.\ \textbf{42} (2016), 81-92.

\bibitem{St}
Stichtenoth, H.,
\textit{Algebraic function fields and codes},
Second edition, Grad.\ Texts in Math., 254, Springer-Verlag, Berlin, 2009.

\bibitem{TT}
Takeuchi, D., Tsushima, T.,
\textit{Gauss sums and Van der Geer-Van der Vlugt curves},
Bull.\ Lond.\ Math.\ Soc.\ \textbf{56} (2024), no.\ 2, 602-623.

\bibitem{Ta}
Tatematsu, R.,
\textit{On some criteria for van der Geer-van der Vlugt curves to be maximal or minimal} (in Japanese),
Master's thesis, Kyoto University, January, 2025.
\end{thebibliography}
\end{document}